\documentclass[oneside,11pt]{amsart}

\pdfoutput=1
\usepackage{amsmath}
\usepackage{amsthm}

\usepackage{amsfonts}
\usepackage{amssymb}
\usepackage{amsxtra}

\usepackage[arrow, tips, matrix]{xy}
\SelectTips{cm}{12}
\CompileMatrices

\newtheorem{thm}{Theorem}[section]
\newtheorem{prop}[thm]{Proposition}
\newtheorem{lem}[thm]{Lemma}
\newtheorem{cor}[thm]{Corollary}

\theoremstyle{definition}
\newtheorem{defn}[thm]{Definition}
\newtheorem{rem}[thm]{Remark}

\newcommand{\abs}[1]{\lvert{#1}\rvert}

\renewcommand{\bar}[1]{\overline{#1}}
\newcommand{\boundary}{\partial}

\newcommand{\presentation}[2]{\langle\, {#1} \mid {#2} \,\rangle}

\newcommand{\set}[2]{\{\,{#1} \mid {#2} \,\}}
\newcommand{\bigset}[2]{ \bigl\{ \, {#1} \bigm| {#2} \, \bigr\} }
\renewcommand{\emptyset}{\varnothing}

\newcommand{\field}[1]{\mathbb{#1}}

\newcommand{\R}{\field{R}}

\newcommand{\N}{\field{N}}

\newcommand{\Hyp}{\field{H}}

\newcommand{\PP}{\field{P}}

\newcommand{\inclusion}{\hookrightarrow}
\newcommand{\of}{\circ}

\DeclareMathOperator{\Isom}{Isom}
\DeclareMathOperator{\CAT}{CAT}

\DeclareMathOperator{\Stab}{Stab}

\newcommand{\ball}[2]{B ( {#1}, {#2} )}

\newcommand{\nbd}[2]{\mathcal{N}_{#2}({#1})}  
\newcommand{\bignbd}[2]{\mathcal{N}_{#2} \bigl( {#1} \bigr)}
\newcommand{\Set}[1]{\mathcal{#1}}

\DeclareMathOperator{\Cayley}{Cayley}
\DeclareMathOperator{\diam}{diam}
\DeclareMathOperator{\Sat}{Sat} 
\DeclareMathOperator{\Hull}{Hull}

\DeclareMathOperator{\SL}{SL}

\DeclareMathOperator{\join}{join}

\def\RomanianComma#1{\setbox0=\hbox{#1}{\ooalign{\hidewidth
    \lower1.2ex\hbox{$\mspace{1mu}^{,}$}\hidewidth\crcr\unhbox0}}}
\newcommand{\Drutu}{Dru{\RomanianComma{t}u}}

\usepackage{ifthen}

\newcommand{\showcomments}{yes}

\newsavebox{\commentbox}
{\ifthenelse{\equal{\showcomments}{yes}}%
{\footnotemark
    \begin{lrbox}{\commentbox}
    \begin{minipage}[t]{1.25in}\raggedright\sffamily\upshape\tiny
    \footnotemark[\arabic{footnote}]}
{\begin{lrbox}{\commentbox}}}%
{\ifthenelse{\equal{\showcomments}{yes}}%
{\end{minipage}\end{lrbox}\marginpar{\usebox{\commentbox}}}
{\end{lrbox}}}

\usepackage{pinlabel}

\begin{document}

\title[Relative hyperbolicity for countable groups]{Relative
hyperbolicity and relative quasiconvexity for countable groups}

\author{G. Christopher Hruska$^{\dag}$}
\address{Dept.\ of Mathematical Sciences\\
University of Wisconsin--Milwaukee\\
P.O.~Box 413\\
Milwaukee, WI 53201\\
USA}
\email{chruska@uwm.edu}
\thanks{$^{\dag}$ Research supported by NSF grants DMS-0505659 and DMS-0731759.}

\date{\today}

\begin{abstract}
We lay the foundations for the study of relatively quasiconvex subgroups
of relatively hyperbolic groups.  These foundations require that we first
work out a coherent theory of countable relatively hyperbolic groups
(not necessarily finitely generated).  We prove the equivalence
of Gromov, Osin, and Bowditch's definitions of relative hyperbolicity
for countable groups.

We then give several equivalent definitions of relatively quasiconvex subgroups
in terms of various natural geometries on a relatively hyperbolic group.
We show that each relatively quasiconvex subgroup is itself
relatively hyperbolic, and that the intersection
of two relatively quasiconvex subgroups is again relatively quasiconvex.
In the finitely generated case, we prove that every undistorted
subgroup is relatively quasiconvex,
and we compute the distortion of a finitely generated
relatively quasiconvex subgroup.
\end{abstract}


\maketitle

\section{Introduction}
\label{sec:Introduction}

Gromov introduced the theory of hyperbolic groups in \cite{Gromov87}.
In this theory, the quasiconvex subgroups play a central
role.  They are geometrically the simplest and most natural subgroups:
those whose intrinsic geometry (the word metric with respect to a
finite generating set) is preserved under the embedding
into the hyperbolic group.
The two most fundamental properties of quasiconvex subgroups of a hyperbolic
group are the following.
\begin{enumerate}
\item Each quasiconvex subgroup is itself word hyperbolic.
\item The intersection of two quasiconvex subgroups is quasiconvex.
\end{enumerate}

In the same article, Gromov also introduced relatively hyperbolic groups.
In the present article, we lay the foundations for a theory of
``relatively quasiconvex'' subgroups, which are expected to play
a central role in the theory of relatively hyperbolic groups.
There are several equivalent ways to formulate this idea.
Two of these formulations were introduced by Dahmani and Osin
\cite{Dahmani03Combination,Osin06} in special cases.
Indeed Osin's definition has already been used in work of
Mart{\'\i}nez-Pedroza \cite{MartinezPedroza09,MartinezPedroza_QCNote}
and Manning--Mart{\'\i}nez-Pedroza \cite{ManningMartinezPedroza_Separation}.

Yet Dahmani's and Osin's definitions were not previously known to be 
equivalent, and neither of these notions was known to satisfy both of the
fundamental properties above.
In addition, no criterion was previously known for relative quasiconvexity
using the ``intrinsic'' geometry of geodesics in the Cayley graph
(for a finite generating set).
In this article, we give several definitions of relative quasiconvexity,
prove their equivalence, and use them to prove
many properties of such subgroups.
In particular, we show that relatively quasiconvex subgroups
satisfy analogues of the two fundamental properties listed above.

\subsection{Non--finitely generated groups and relative hyperbolicity}
Historically
the main conceptual difficulty to forming a satisfying theory of relative
quasiconvexity has been the necessity of considering
non--finitely generated groups.  Relatively hyperbolic groups have been defined
in different ways by Gromov, Farb, and \Drutu--Sapir
\cite{Gromov87,Farb98,DrutuSapirTreeGraded}.
Three different model geometries arise in these three definitions.
The geometry in Gromov's definition can be obtained by attaching
``horoballs'' to the peripheral subgroups.
This definition generalizes the fundamental group of a finite volume
hyperbolic manifold.
The geometry in Farb's definition (of strong relative hyperbolicity)
is obtained by collapsing each peripheral subgroup to a set of bounded diameter.
This definition generalizes the structure of a free product
acting on its Bass--Serre tree.
The geometry in \Drutu--Sapir's definition is the ``intrinsic''
geometry of the word metric with respect to a finite generating set.
This definition generalizes the geometry of a $\CAT(0)$ space with isolated 
flats.
These model geometries have substantially different flavors.

Farb's and \Drutu--Sapir's definitions each require that the group in question
be finitely generated.  However, Gromov's definition
requires only that the group be countable.
(It must act properly discontinuously on a proper metric space.)
In the setting of finitely generated groups, these three definitions are
equivalent. (See Bowditch \cite{BowditchRelHyp} and \Drutu--Sapir 
\cite{DrutuSapirTreeGraded} for details.)

Bowditch and Osin have given variants of Farb's definition that do not
require finitely generated groups \cite{BowditchRelHyp,Osin06}.
(They are variants in the sense
that they essentially use the same model geometry introduced by Farb.)
The exact relation between these definitions has not been clear.
Many researchers have concluded that finite generation should be part
of the definition of relative hyperbolicity.
(Gromov and Osin are notable exceptions to this trend.)
Indeed, no version of \Drutu--Sapir's definition is known
for non--finitely generated groups.

Yet any natural definition of relatively quasiconvex subgroups
will include non--finitely generated groups.
Additionally, the intersection of two finitely generated relatively
quasiconvex subgroups is often not finitely generated.
Thus in order to formulate a natural theory of relative hyperbolicity,
we must allow non--finitely generated relatively hyperbolic groups.

The following theorem reconciles the various existing definitions
into a unified theory of relative hyperbolicity that includes all
of the above examples.  See Section~\ref{sec:RelHypDefn} for precise
statements of the various definitions of relative hyperbolicity.

\begin{thm}
If $G$ is countable and $\PP$ is a finite collection of infinite
subgroups of $G$, then the definitions of relative hyperbolicity
for $(G,\PP)$ given by Gromov, Bowditch, and Osin are equivalent.
\end{thm}

In fact, there are numerous examples of non--finitely generated
relatively hyperbolic groups:
\begin{itemize}
\item A free product $A*B$ where $A$ or $B$ is not finitely generated
is relatively hyperbolic with respect to the factors.
\item A nonuniform lattice $\Gamma$ in a rank one Lie
group $G$ over a nonarchimedian local field is relatively hyperbolic with respect
to its parabolic subgroups by Lubotzky \cite{Lubotzky91}.
These lattices are never finitely generated.
A typical example is the lattice
$\Gamma = \SL_2 \bigl( \mathbb{F}_q[t] \bigr)$
in the group $G = \SL_2 \bigl( \mathbb{F}_q((1/t)) \bigr)$.
In this case $\Gamma$ acts on the
Bruhat--Tits tree for $G$ with quotient a ray.
\item Many relatively quasiconvex subgroups of relatively hyperbolic groups
are not finitely generated.
\end{itemize}

\subsection{Relative quasiconvexity}
As mentioned above, two different definitions of a relatively
quasiconvex subgroup $H$ of a relatively hyperbolic group $G$
have been introduced by Dahmani and Osin
in the special case that $G$ is finitely generated.
Dahmani's ``dynamical quasiconvexity'' is defined in terms of
the dynamics at infinity in Gromov's model
geometry.  Osin's ``relative quasiconvexity'' is defined in terms
of Farb's ``coned--off'' model geometry and makes sense only for
subgroups of a finitely generated relatively hyperbolic group.
Dahmani's definition immediately implies that a dynamically
quasiconvex subgroup is itself relatively hyperbolic.
On the other hand, Mart{\'\i}nez-Pedroza
used Osin's definition to prove that the intersection of two relatively
quasiconvex subgroups is again relatively quasiconvex
\cite{MartinezPedroza09}.
Osin asked whether the two definitions are equivalent, and also asked whether
all relatively quasiconvex subgroups are relatively hyperbolic
(using his definitions) \cite{Osin06}.

In this article, we clarify the notion of relatively quasiconvex subgroups
by giving criteria for relative quasiconvexity
in terms of the model geometries of Gromov, Farb, and \Drutu--Sapir.
(Recall that \Drutu--Sapir's geometry, the word metric for a finite generating set, is defined only when $G$ is finitely generated.)
When $G$ is finitely generated, these subgroups coincide with those studied by Dahmani and Osin.
A significant part of this paper involves showing that
various definitions of relative quasiconvexity are equivalent.
This equivalence has several consequences.

\begin{thm}
\label{thm:QCBasicProps}
Let $G$ be a countable group that is relatively hyperbolic with respect
to a finite family of subgroups $\PP=\{P_1,\dots,P_n\}$.
\begin{enumerate}
\item If $H\le G$ is relatively quasiconvex, then $H$ is relatively
hyperbolic with respect to a natural induced collection of subgroups.
\item If $H_1,H_2 \le G$ are relatively quasiconvex, then
$H_1 \cap H_2$ is also relatively quasiconvex.
\end{enumerate}
\end{thm}

In the manifold setting, our characterization of relative quasiconvexity
has the following corollary.

\begin{cor}
\label{cor:GeomFiniteManifolds}
Let $G$ be a geometrically finite subgroup of $\Isom(X)$
for $X$ a complete, simply connected manifold with pinched negative curvature.
Then $H\le G$ is relatively quasiconvex
\textup{(}with respect to the maximal parabolic
subgroups of $G$\textup{)} if and only if $H$ is geometrically finite.

Furthermore, if $H,K \le G$ are geometrically finite then $H \cap K$
is also geometrically finite.
\end{cor}

In the special case when $X$ is the real hyperbolic space $\Hyp^n$,
the second claim of the preceding corollary
is due to Susskind--Swarup \cite{SusskindSwarup92}.
See also Corollary~\ref{cor:GeomFiniteKleinian} below,
which strengthens the first conclusion in the real hyperbolic case.
The preceding corollary appears to be new even in the special case of
complex hyperbolic manifolds.

If $G$ is relatively hyperbolic, $H\le G$ is relatively quasiconvex,
and both $H$ and $G$ are finitely generated, we compute
the distortion of $H$ in $G$,
which measures the difference between the word metric
on $H$ and the word metric on $G$.  See Theorem~\ref{thm:Distortion}
for a more precise statement.

\begin{thm}
\label{thm:DistortionVague}
Let $G$ be relatively hyperbolic with respect to $P_1,\dots,P_n$.
Suppose $H \le G$ is relatively
quasiconvex, and both $H$ and $G$ are finitely generated.
Then the distortion
of $H$ in $G$ is a combination of the distortions
of the infinite subgroups $O\le P_i$ where $O=gHg^{-1} \cap P_i$.

More precisely, the distortion $\Delta_H^G$ of $H$ in $G$ satisfies
\[
   f \preceq \Delta_H^G \preceq \bar{f}
\]
where $f$ is the supremum of the distortions of the
subgroups $O \le P_i$ and $\bar{f}$ is the superadditive closure
of $f$.
\end{thm}

The proof of this theorem uses a characterization of relative quasiconvexity
in terms of the word metric on $G$ with respect to a finite generating set.
Indeed Gromov's and Farb's model geometries are poorly suited to prove such a
result, since the subgroups $P_i$ are badly distorted in these geometries
(exponentially distorted in the Gromov model and
crushed to a finite diameter in the Farb model).
On the other hand, the proof is quite natural using the word
metric because by \Drutu--Sapir \cite{DrutuSapirTreeGraded}
the peripheral subgroups $P_i$ are undistorted in $G$.

\Drutu--Sapir previously showed that undistorted subgroups
of a finitely generated relatively hyperbolic group
are themselves relatively hyperbolic \cite{DrutuSapirTreeGraded}.
The following theorem places \Drutu--Sapir's result in the context
of relative quasiconvexity.

\begin{thm}[Undistorted $\Longrightarrow$ relatively quasiconvex]
\label{thm:UndistortedStatement}
Let $G$ be a finitely generated relatively hyperbolic group
and let $H$ be a finitely generated subgroup.
If $H$ is undistorted in $G$, then $H$ is relatively quasiconvex.
\end{thm}

Using Theorem~\ref{thm:DistortionVague} and a result of the author
from \cite{HruskaGeometric}, we obtain a
characterization of the geometrically finite subgroups $H$
of a geometrically finite Kleinian group $G$
that strengthens the conclusion of Corollary~\ref{cor:GeomFiniteManifolds}.

\begin{cor}
\label{cor:GeomFiniteKleinian}
Let $G$ be a geometrically finite subgroup of $\Isom(\Hyp^n)$,
and let $H\le G$ be a subgroup. The following are equivalent.
\begin{enumerate}
\item $H$ is geometrically finite.
\item $H$ is relatively quasiconvex with respect to the maximal parabolic
subgroups of $G$.
\item $H$ is finitely generated and undistorted in $G$,
in the sense that the inclusion $H \inclusion G$ is a quasi-isometric
embedding with respect to the word metrics for finite generating sets.
\item $H$ is $\CAT(0)$--quasiconvex in $G$, in the sense that
whenever $G$ acts properly, cocompactly, and isometrically on any
$\CAT(0)$ space $X$, the orbits of $H$ are quasiconvex in $X$.
\end{enumerate}
\end{cor}

\subsection{Summary of the sections}
In Section~\ref{sec:Horoballs} we review the definition of a horoball
in a $\delta$-hyperbolic space and prove several lemmas about geodesics
in horoballs.
In Section~\ref{sec:RelHypDefn} we explicitly state six different
definitions of relative hyperbolicity.
Section~\ref{sec:CuspedSpace} is a review of
Groves--Manning's construction of combinatorial horoballs based on a
connected graph.  We observe that their construction can be applied
more generally to produce horoballs based on an arbitrary metric
space.  Using this observation, we state more general versions of some
results of Groves--Manning that we apply in the following section.
Section~\ref{sec:RelHypEquiv} consists of the proof
that the six definitions of relative hyperbolicity are equivalent.
Many of the directions are proved by observing that various proofs in the literature go through without change when one uses weaker hypotheses.

In Section~\ref{sec:QCDef} we introduce five equivalent definitions
of relative quasiconvexity for subgroups of a
relatively hyperbolic group. Section~\ref{sec:QCEquivalence}
contains a proof of the equivalence of the definitions introduced in the 
previous section.
In Section~\ref{sec:QCWordMetric} we turn our attention to the
word metric on a relatively hyperbolic group $G$ with a finite generating set.
In the case when $G$ is finitely generated, we characterize relatively
quasiconvex subgroups $H$ in terms of the word metric on $G$.

In Section~\ref{sec:applications} we prove several basic results
about relatively quasiconvex subgroups,
including Theorems \ref{thm:QCBasicProps} and~\ref{thm:UndistortedStatement}
and Corollary~\ref{cor:GeomFiniteManifolds}.
We also characterize strongly relatively quasiconvex subgroups,
which were introduced by Osin in \cite{Osin06}.
Finally in Section~\ref{sec:SubgroupInclusion}
we examine distortions of relatively quasiconvex subgroups,
proving Theorem~\ref{thm:Distortion} (Theorem~\ref{thm:DistortionVague})
and Corollary~\ref{cor:GeomFiniteKleinian}.

\subsection{Remark}
Shortly after the initial circulation of this article,
Agol--Groves--Manning \cite{AgolGrovesManning09}
introduced another definition of a relatively quasiconvex subgroup.
Manning and Mart{\'\i}nez-Pedroza \cite{ManningMartinezPedroza_Separation}
have shown that
this definition is equivalent to definition QC-3 of the present article.

\section{Hyperbolic spaces and horoballs}
\label{sec:Horoballs}

A geodesic space $(X,\rho)$ is \emph{$\delta$--hyperbolic} if every geodesic
triangle with vertices in $X$ is $\delta$--thin
in the sense that each side lies in the $\delta$--neighborhood
of the union of the other two sides.
If $X$ is a $\delta$--hyperbolic space, the \emph{boundary} (or \emph{boundary at infinity}) of $X$,
denoted $\boundary X$, is the set of all equivalence classes of
geodesic rays in $X$, where two rays $c,c'$ are equivalent if
the distance $d\bigl(c(t),c't)\bigr)$ remains bounded as $t \to \infty$.
The set $\boundary X$ has a natural topology, which is compact
and metrizable (see, for instance, Bridson--Haefliger \cite{BH99} for details).

Increasing the constant $\delta$ if necessary, we will also assume
that every geodesic triangle with vertices in $X \cup \boundary X$
is $\delta$--thin.

Let $\Delta=\Delta(x,y,z)$ be a triangle
with vertices $x,y,z \in X \cup \boundary X$.
A \emph{center} of $\Delta$ is a point $w \in X$ such that
the ball $\ball{w}{\delta}$ intersects all three sides of the triangle.
It is clear that each side of $\Delta$ contains a center of $\Delta$.

A function $h \colon X \to \R$ is a
\emph{horofunction} about a point $\xi \in \boundary X$
if there is a constant $D_0$ such that the following holds.
Let $\Delta(x,y,\xi)$ be a geodesic triangle, and let $w \in X$ be a center
of the triangle.
Then 
\[
   \left| \bigl( h(x) + \rho(x,w) \bigr) - \bigl( h(y) - \rho(y,w) \bigr)
   \right|  < D_0.
\]
A closed subset $B \subseteq X$ is a \emph{horoball centered at $\xi$}
if there is a horofunction $h$ about $\xi$ and a constant $D_1$
such that $h(x) \ge - D_1$ for all $x \in B$ and $h(x) \le D_1$
for all $x \in X - B$.

We remark that every horoball or horofunction has a unique center,
and every $\xi \in \boundary X$ is the center of a horofunction
(see Gromov \cite[Section~7.5]{Gromov87} for details).

\begin{lem}\label{lem:NearHoroballComplement}
Let $B$ be a horoball of a $\delta$--hyperbolic space $(X,\rho)$.
For each $L>0$, there is a constant $M_0=M_0(B,L)$
such that any geodesic $c$ in $\nbd{X-B}{L} \cap B$ has length at most $M$.
\end{lem}

The notation $\nbd{A}{r}$ denotes the open $r$--neighborhood of $A$;
i.e., the set of all points at a distance less than $r$ from $A$.

\begin{proof}
Let $\xi \in \boundary X$ be the center of $B$, and choose a horofunction
$h$ and constants $D_0, D_1$ for $B$ as above.
Suppose $c$ has endpoints $x_0$ and $x_1$, and choose geodesic rays
$[x_i,\xi)$ for $i=1,2$.
Choose $z \in c$ within a distance $2\delta$ of both rays $[x_i,\xi)$,
and choose a ray $[z,\xi)$.

In order to complete the proof, it suffices to bound $\rho(x_i,z)$ from above.
We first compute an upper bound for $h(z)$.
If we choose $w \in X-B$ such that $\rho(w,z) < L$, then $h(w) \le D_1$.
Applying the definition of horofunction to a triangle with vertices
$w,z,\xi$ and center $p \in [w,z]$,
gives
\[
   h(z) + \rho(z,p) < h(w) + \rho(z,p) + D_0.
\]
Thus we obtain the following upper bound:
\begin{align*}
   h(z) &< h(w) + \rho(w,p) - \rho(z,p) + D_0 \\
        &\le h(w) + \rho(w,z) + D_0 \\
        &\le D_1 + L + D_0.
\end{align*}
On the other hand, since $x_i \in B$ we have a lower bound $h(x_i) \ge -D_1$.

Let $\Delta_i=\Delta(x_i,z,\xi)$ be the triangle with sides
$[x_i,\xi)$, $[z,\xi)$, and $[x_i,z]$, where $[x_i,z]$ is the portion of
$c$ from $x_i$ to $z$.
Observe that $z$ is a center of $\Delta_i$.
Applying the definition of horofunction to $\Delta_i$, we have
\[
   \rho(x_i,z) < h(z) - h(x_i) + D_0
             \le (D_1 + L + D_0) + D_1 + D_0,
\]
completing the proof.
\end{proof}

The following result can be also proved by a similar proof.

\begin{lem}\label{lem:BetweenHoroballs}
Let $B$ and $B'$ be horoballs of $(X,\rho)$ centered at the same point
$\xi \in \boundary X$.
There is a constant $M_1=M_1(B,B')$
such that any geodesic $c$ in $B-B'$ has length at most $M_1$.\qed
\end{lem}

\begin{lem}\label{lem:GeodesicInHoroball}
Let $B$ be a horoball of $(X,\rho)$ with corresponding horofunction $h$.
There is a constant $M_2 = M_2(B,h)$ such that the following holds.
Suppose $x,y \in X-B$ satisfy $h(x)=h(y)=0$.
Then for any geodesic $c$ joining $x$ and $y$ we have
\[
   c \cap (X-B) \subseteq \bignbd{\{x,y\}}{M_2}.
\]
\end{lem}

\begin{proof}
Let $z$ be an arbitrary point of $c \cap (X-B)$.
Then $h(z) \le D_1$.
Choose rays $[x,\xi)$, $[y,\xi)$ and $[z,\xi)$.
The point $z$ is within a distance $2\delta$ of one of the rays $[x,\xi)$
or $[y,\xi)$, say $[x,\xi)$.
Since $z$ is a center of $\Delta=\Delta(x,z,\xi)$,
the definition of horofunction implies that
\[
   \rho(x,z) \le h(z) + D_0 \le D_1 + D_0,
\]
completing the proof.
\end{proof}

\section{Notions of relative hyperbolicity for countable groups}
\label{sec:RelHypDefn}

It is well-known that finitely generated relatively hyperbolic groups can be 
characterized in
several equivalent ways.  In this section, we discuss natural extensions
of these properties to the setting of countable groups.
Throughout the section, $G$ is a countable group with a finite
collection of subgroups $\PP=\{P_1,\dots,P_n\}$.

\subsection{Geometrically finite groups}
\label{subsec:GeomFinite}

The first definition of relative hyperbolicity
is in terms of dynamical properties of an action on a compact
space $M$.
More precisely, below we define a group $G$ to be relatively hyperbolic
if it admits a geometrically finite convergence group action on some
compactum $M$.  In fact by Corollary~\ref{cor:RH12} below the compactum $M$
always arises as the boundary of a $\delta$--hyperbolic space
on which $G$ acts.
Thus we lose no generality by keeping this geometric example
in mind throughout the following definitions.

The notion of a convergence group action was introduced by
Gehring--Martin in \cite{GehringMartin87} to axiomatize certain 
dynamical properties of the action of a Kleinian group on its limit set
in the ideal sphere at infinity of real hyperbolic space.
The dynamical version of geometrical finiteness defined here
was introduced by Beardon--Maskit \cite{BeardonMaskit74}
for Kleinian groups.

A \emph{convergence group action} is an action
of a group $G$ on a compact, metrizable
space $M$ satisfying the following conditions, depending on the cardinality
of $M$:
\begin{itemize}
\item If $M$ is the empty set, then $G$ is finite.
\item If $M$ has exactly one point, then $G$ can be any countable group.
\item If $M$ has exactly two points, then $G$ is virtually cyclic.
\item If $M$ has at least three points, then the action of $G$ on the space
of distinct (unordered) triples of points of $M$ is properly discontinuous.
\end{itemize}
In the first three cases the action is \emph{elementary}, and in the final
case the action is \emph{nonelementary}.

Suppose $G$ has a convergence group action on $M$.
An element $g \in G$ is \emph{loxodromic} if it has infinite order
and fixes exactly two points of $M$.
A subgroup $P \le G$ is a \emph{parabolic subgroup} if it is infinite and
contains no loxodromic element.
A parabolic subgroup $P$ has a unique fixed point in $M$, called a
\emph{parabolic point}.
The stabilizer of a parabolic point is always a maximal parabolic group.
A parabolic point $p$ with stabilizer $P := \Stab_G(p)$ is \emph{bounded}
if $P$ acts cocompactly on $M - \{p\}$.
A point $\xi \in M$ is a \emph{conical limit point} if
there exists a sequence $(g_i)$ in $G$ and distinct points
$\zeta_0,\zeta_1 \in M$
such that $g_i(\xi) \to \zeta_0$, while for all $\eta \in M - \{\xi\}$
we have $g_i(\eta) \to \zeta_1$.

Tukia has shown that every properly discontinuous action of a group $G$
on a proper $\delta$--hyperbolic space induces a convergence group
action on the boundary at infinity \cite{Tukia94} (a similar result
was proved independently by Freden \cite{Freden95}).

A convergence group action of $G$ on $M$ is \emph{geometrically finite}
if every point of $M$ is either a conical limit point or a
bounded parabolic point.
In addition if $\PP$ is a set of representatives of the conjugacy classes
of maximal parabolic subgroups, then we say that
the action of the pair $(G,\PP)$ on $M$ is geometrically finite.
Observe that every elementary convergence group action is geometrically finite.

The following definition was proposed by Bowditch in \cite{BowditchRelHyp} and 
studied by Yaman in \cite{Yaman04} (with the additional assumption that
the peripheral subgroups $P \in \PP$ are finitely generated).

\begin{defn}[RH-1]
Suppose $(G,\PP)$ has a geometrically finite convergence group action
on a compact, metrizable space $M$.
Then $(G,\PP)$ is \emph{relatively hyperbolic}.
\end{defn}

The following definition is similar to the preceding one, except that we assume
that the compact space $M$ is the boundary of a $\delta$--hyperbolic space.
This definition was introduced by Bowditch in \cite{BowditchRelHyp}.

\begin{defn}[RH-2]
Suppose $G$ has a properly discontinuous action on a proper
$\delta$--hyperbolic space $X$
such that the induced convergence group action on $\boundary X$ is
geometrically finite.
If $\PP$ is a set of representatives of the conjugacy classes of maximal 
parabolic subgroups then $(G,\PP)$ is \emph{relatively hyperbolic}.

In this case, we also say that $(G,\PP)$ acts \emph{geometrically finitely}
on $X$.
\end{defn}

\subsection{Cusp uniform actions}
\label{subsec:CuspUniformActions}

The next definition is essentially Gromov's original definition of relative 
hyperbolicity from \cite{Gromov87},
as stated by Bowditch in \cite{BowditchRelHyp}.
The structure is analogous to the well-known decomposition 
of a finite volume hyperbolic manifold as the union of a compact part
together with finitely many cusps
due to Garland--Raghunathan \cite{GarlandRaghunathan70}
(see also Thurston \cite[Section~4.5]{Thurston97}).

\begin{defn}[RH-3]
Suppose $G$ acts properly discontinuously on a proper $\delta$--hyperbolic 
space $X$, and $\PP$ is a set of representatives of the conjugacy classes 
of maximal parabolic subgroups.
Suppose also that there is a $G$--equivariant collection of disjoint
horoballs centered at the parabolic points of $G$, with union $U$ open in $X$,
such that the quotient of $X-U$ by the action of $G$ is compact.
Then $(G,\PP)$ is \emph{relatively hyperbolic}.

In this case, we say that the action of $(G,\PP)$ on $X$
is \emph{cusp uniform},
and the space $Y=X-U$ is a \emph{truncated space} for the action.
By a slight abuse of notation, we refer to the horoballs of $U$
as \emph{horoballs of $Y$}.
If $U'$ is any other $G$--equivariant family of disjoint open horoballs
centered at the parabolic points of $X$, then $G$ also acts cocompactly
on $Y'=X-U'$, and hence $Y'$ is also a truncated space for the cusp uniform
action of $(G,\PP)$ on $X$ (see Bowditch
\cite{BowditchRelHyp} for details).
\end{defn}

\subsection{Fine hyperbolic graphs}
\label{subsec:FineHyperbolicGraphs}

The following definition of relative hyperbolicity
was proposed by Bowditch in \cite{BowditchRelHyp}
as an abstraction of the Farb approach that does not require finite generation.
Bowditch studied finitely generated groups satisfying this condition,
but it is a completely trivial matter to remove the finite generation 
hypothesis from his definition.

A graph $K$ is \emph{fine} if each edge of $K$ is contained in
only finitely many circuits of length $n$ for each $n$.

\begin{defn}[RH-4]
Suppose $G$ acts on a $\delta$--hyperbolic graph $K$ with finite edge stabilizers
and finitely many orbits of edges.  If $K$ is fine, and $\PP$
is a set of representatives of the conjugacy classes of infinite vertex stabilizers
then $(G,\PP)$ is \emph{relatively hyperbolic}.
\end{defn}

\subsection{The coned-off Cayley graph and Bounded Coset Penetration}
\label{subsec:BCP}

Next we consider Farb's notion of (strong)
relative hyperbolicity from \cite{Farb98}.
(Our terminology here does not agree with Farb's and has been adjusted to be
more consistent with the other definitions of relative hyperbolicity
in this article.)

In brief, Farb's definition requires that the ``coned-off'' Cayley graph is
hyperbolic and satisfies Bounded Coset Penetration (defined below).
Farb also considered a weak version of relative hyperbolicity,
which has many interesting applications.  We will not discuss
weak relative hyperbolicity here.

Farb originally proposed this definition for finitely generated groups.
After the work of Bowditch and Osin \cite{BowditchRelHyp,Osin06},
it is clear that the following condition is the natural
formulation of Farb's notion in the non--finitely generated setting.

Let $G$ be a group with a collection of subgroups $\PP = \{P_1,\dots,P_n\}$.
A set $\Set{S}$ is a \emph{relative generating set} for the pair
$(G,\PP)$ if the set $\Set{S} \cup P_1 \cup \cdots \cup P_n$ is a generating
set for $G$ in the traditional sense.
We will always implicitly assume that $\Set{S}$ is symmetrized, so that
$\Set{S} = \Set{S}^{-1}$.
Let $\Gamma$ be the Cayley graph $\Cayley(G,\Set{S})$.
Note that $\Gamma$ is connected if and only if $\Set{S}$ is a traditional
generating set for $G$.
We do not require connectedness of $\Gamma$.

Form a new graph $\hat\Gamma(G,\PP,\Set{S})$, called the \emph{coned-off
Cayley graph}, as follows.
For each left coset $gP$ with $g \in G$ and $P \in \PP$,
add a new vertex $v(gP)$ to $\Gamma$, and add an edge of length $1/2$
from this new vertex to each element of $gP$.
The key point to note here is that the coned-off Cayley graph is
connected if and only if $\Set{S}$ is a relative generating set
for $(G,\PP)$.

Given an oriented path $\gamma$
in the coned-off Cayley graph, we say that $\gamma$
\emph{penetrates} the coset $gP$ if $\gamma$ passes through the cone point
$v(gP)$.
A vertex $v_i$ of $p$ immediately preceding the cone point is an
\emph{entering vertex} of $p$ in the coset $gP$.
Exiting vertices are defined similarly.
Observe that entering and exiting vertices are always elements of $G$,
and hence can also be considered as vertices of $\Gamma$.
A path $\gamma$ in the coned-off Cayley graph is \emph{without
backtracking} if, for every coset $gP$ which $\gamma$ penetrates,
$\gamma$ never returns to $gP$ after leaving $gP$.

\begin{defn}[Bounded Coset Penetration]
Let $\Set{S}$ be a relative generating set for $(G,\PP)$.
The triple $(G,\PP,\Set{S})$ satisfies
\emph{Bounded Coset Penetration} if for each $\lambda \ge 1$, there is a 
constant $a(\lambda) > 0$ such that if $\gamma$ and $\gamma'$
are $(\lambda,0)$--quasigeodesics without backtracking in
the coned-off Cayley graph
$\hat\Gamma(G,\PP,\Set{S})$
with initial endpoints $\gamma_-=\gamma'_-$ and terminal endpoints
$\gamma_+$ and $\gamma'_+$, and $d_{\Set{S}}(\gamma_+,\gamma'_+) \le 1$,
then the following two conditions hold.
\begin{enumerate}
\item If $\gamma$ penetrates a coset $gP$, but $\gamma'$ does not
penetrate $gP$, then the entering vertex and exiting vertex of $\gamma$
in $gP$ are at an $\Set{S}$--distance at most $a$ from each other.
\item If $\gamma$ and $\gamma'$ both penetrate a coset $gP$,
then the entering vertices of $\gamma$ and $\gamma'$ in $gP$
are at an $\Set{S}$--distance at most $a$ from each other.
Similarly, the exiting vertices are at an $\Set{S}$--distance 
at most $a$ from each other.
\end{enumerate}
\end{defn}

\begin{defn}[RH-5]
Suppose $G$ is finitely generated relative to $\PP$,
and each $P_i$ is infinite.
The pair $(G,\PP)$ is \emph{relatively hyperbolic} if
for some (every) finite relative generating set $\Set{S}$,
the coned-off Cayley graph $\hat\Gamma(G,\PP,\Set{S})$
is $\delta$--hyperbolic
and $(G,\PP,\Set{S})$ has Bounded Coset Penetration.
\end{defn}

In the finitely generated case, BCP is independent of the choice of generating 
set due to the following fact: a change of generating set induces
a quasi-isometry of Cayley graphs
and also a quasi-isometry of coned-off Cayley graphs.
However, in general this independence is less obvious.
The difficulty is that a Cayley graph for a nongenerating set is not connected.
It is possible for two group elements to be connected by an edge
in one Cayley graph and lie in distinct components of another.
In Section~\ref{sec:RelHypEquiv} we indicate a proof of this independence
based on arguments of Dahmani \cite{DahmaniThesis}.

\subsection{Linear relative Dehn function}
\label{subsec:DehnFunctions}

Finally we state the definition due to Osin in terms of relative isoperimetric 
functions introduced in \cite{Osin06}.
The definition, as presented here, includes only countable groups $G$ with
finitely many infinite peripheral subgroups $\PP = \{P_1,\dots,P_n\}$.
We remark that Osin's theory does not impose such restrictions.
However, these restrictions are necessary for the equivalence with 
other definitions of relative hyperbolicity.

Let $G$ be a group with a collection of subgroups $\PP$.
If $(G,\PP)$ has a relative generating set $\Set{S}$,
then there is a canonical homomorphism from the group
\[
   K :=F(\Set{S}) * \bigl(*_{P \in \PP} \tilde{P} \bigr)
\]
onto $G$.
Here $F(\Set{S})$ denotes the group freely generated by $\Set{S}$,
and $\tilde{P}$ denotes an abstract group isomorphic to $P$.
Let $N$ be the kernel of this homomorphism.

If $\Set{R} \subseteq N$ is a subset whose normal closure in $K$ is equal
to $N$, then we say that $G$ has a \emph{relative presentation}
\[
   \presentation{\PP,\Set{S}}{\Set{R}}.
\]
If $\Set{S}$ and $\Set{R}$ are both finite, then the relative presentation
above is a \emph{finite relative presentation}.

Suppose $(G,\PP)$ has a relative presentation as above.
Consider the disjoint union
\[
   \Set{P} := \coprod_{P \in \PP} \bigl( \tilde{P} - \{1\} \bigr).
\]
Each word $W$ in the alphabet $(\Set{S} \cup \Set{P})$
naturally represents both an element of $K$ and an element of $G$
(via the quotient map above).
If $W$ represents $1$ in $G$, then in $K$ we have an equation
\begin{equation}
\tag{\dag}
\label{eqn:Dehn}
   W =_{K} \prod_{i=1}^\ell f_i^{-1} R_i f_i
\end{equation}
where $R_i \in \Set{R}$ and $f_i \in K$ for each $i$.

A function $f \colon \N \to \N$ is a \emph{relative isoperimetric function}
for the relative presentation $\presentation{\PP,\Set{S}}{\Set{R}}$
if for each $m \in \N$ and any word $W$ over $(\Set{S} \cup \Set{P})$
of length at most $m$ representing the identity of $G$, we have an equation
in $K$ of the form (\ref{eqn:Dehn}) with $\ell \le f(m)$.
We remark that some relative presentations do not admit a
finite relative isoperimetric function.
The \emph{relative Dehn function} of a relative presentation
is the smallest possible relative isoperimetric function.
If there does not exist a relative isoperimetric function with finite values,
we say that the relative Dehn function is not well-defined.

\begin{defn}[RH-6]
Suppose $\PP$ is a finite collection of infinite subgroups of a countable
group $G$.
If $(G,\PP)$ has a finite relative presentation,
and the relative Dehn function is well-defined and linear
for some/every finite relative presentation
then $(G,\PP)$ is \emph{relatively hyperbolic}.
\end{defn}

We remark that if one finite relative presentation has a linear relative Dehn function then so does any other by Osin \cite[Theorem~2.34]{Osin06}.

\section{The cusped space}
\label{sec:CuspedSpace}

Our immediate goal is to prove that the definitions of relative hyperbolicity
for countable groups given above are equivalent.
As discussed below in Section~\ref{sec:RelHypEquiv}
several of the necessary implications are either already known or follow from
straightforward modifications of existing proofs in the literature.

In this section we lay the groundwork for the implication
(RH-6)~$\Longrightarrow$~(RH-3).
The main step is the construction of a space obtained by attaching certain
``combinatorial horoballs'' to $G$ along the left cosets
of the peripheral subgroups.
The idea of gluing horoballs onto a finitely generated group
is due to Cannon--Cooper \cite{CannonCooper92},
and variations have been studied by Bowditch, Rebbechi, and Groves--Manning
\cite{BowditchRelHyp,Rebbechi01,GrovesManning08}.

The construction presented here is closely modeled on Groves--Manning's
combinatorial horoballs.
Groves--Manning constructed horoballs based on an arbitrary connected
graph, and showed that such a horoball is always $\delta$--hyperbolic.
A minor modification of the Groves--Manning construction produces
connected horoballs based on any metric space.
If the base metric space is discrete and proper,
the resulting horoball is a locally finite $\delta$--hyperbolic graph.

Using this modification,
Groves--Manning's proofs about finitely generated
relatively hyperbolic groups extend to any countable relatively hyperbolic
group equipped with a proper, left invariant metric.

\begin{defn}[Combinatorial horoballs]
Let $(P,d)$ be a metric space.
The \emph{combinatorial horoball} based on $P$, denoted $\mathcal{H}(P,d)$
is the graph defined as follows:
\begin{enumerate}
\item $\mathcal{H}^{(0)} = P \times \N$.
\item $\mathcal{H}^{(1)}$ contains the following two types of edges
   \begin{enumerate}
   \item For each $k\in \N$ and $p,q\in P$, if $0<d(p,q)\le2^k$
   then there is a \emph{horizontal edge} connecting $(p,k)$ and $(q,k)$.
   \item For each $k\in \N$ and $p\in P$, there is a \emph{vertical edge}
   connecting $(p,k)$ and $(p,k+1)$.
   \end{enumerate}
\end{enumerate}
\end{defn}

As in Groves--Manning \cite{GrovesManning08},
we consider each edge of $\mathcal{H}$ to have
length one, and we endow $\mathcal{H}$ with the induced length metric $d_{\mathcal{H}}$.
We identify $P$ with its image $P\times\{0\}$ in $\mathcal{H}$.

The following result is proved by Groves--Manning under the additional hypothesis that $(P,d)$
is the set of vertices of a connected metric graph where each edge has length one
\cite[Section~3.1]{GrovesManning08}.
The proof given there extends with no changes to the current setting.

\begin{thm}
Let $(P,d)$ be any metric space \textup{(}with all distances
finite\textup{)}.
Then $\mathcal{H}(P,d)$ is connected and $\delta$--hyperbolic for some
$\delta\ge 0$ independent of $(P,d)$.
Moreover, if $(P,d)$ is a discrete, proper metric space,
then $\mathcal{H}(P,d)$ is a locally finite graph.
\qed
\end{thm}

\begin{defn}[The augmentated space]
Suppose $G$ is a countable group with subgroups $\PP=\{P_1,\dots,P_n\}$.
Let $\Set{S}$ be a finite relative generating set for $(G,\PP)$,
and choose a proper, left invariant metric $d_i$ on each $P_i$.
The metric $d_i$ induces a metric (again called $d_i$) on each left coset
$gP_i$ using left translation by $g$.
The metric space $(gP_i,d_i)$ is locally finite.
Consequently, $\mathcal{H}(gP_i,d_i)$ is a locally finite connected graph.

Let $\Gamma$ be the graph $\Cayley(G,\Set{S})$.
The \emph{augmented space} is the space
\[
   X := \Gamma \cup \left( \bigcup \mathcal{H}(gP_i,d_i)^{(1)} \right),
\]
where the inner union ranges over all $i=1,\dots,n$ and all left cosets
$gP_i$ of $P_i$,
and where each coset $gP_i \subseteq \Gamma^{(0)}$ is identified with
$gP_i \subset \mathcal{H}(gP_i,d_i)$ in the obvious way.
\end{defn}

Observe that $X$ is locally finite since each combinatorial horoball
is locally finite and the collection of cosets $gP_i$
is locally finite in $G$
(i.e., each finite subset of $G$ intersects only finitely many such cosets).
Furthermore $X$ is connected because $\Set{S}$ is
a relative generating set.
The natural action of $G$ on $X$ is properly discontinuous and isometric.

The next theorem follows from arguments of Groves--Manning
in \cite[Section~3.3]{GrovesManning08}.
Groves--Manning assume that $G$ is finitely generated, but their proofs 
remain valid when $G$ is countable and $X$ is constructed as above.
The only difference here is that we have imposed a proper, left invariant 
metric on each peripheral subgroup, while Groves--Manning metrize each 
peripheral subgroup using the word metric for a finite generating set.

\begin{thm}
\label{thm:AugmentedHyperbolic}
Suppose $G$ is countable and $\PP$ is a finite collection of subgroups
of $G$.
Choose a finite relative generating set $\Set{S}$ for $(G,\PP)$.
For each $P_i \in \PP$ choose a proper, left invariant metric $d_i$.
If $(G,\PP)$ has a finite relative presentation with a
linear relative Dehn function, then the augmented space $X$ is
$\delta$--hyperbolic.
\qed
\end{thm}

\section{Equivalence of notions of relative hyperbolicity for countable 
         groups}
\label{sec:RelHypEquiv}

The goal of this section is the following theorem.

\begin{thm}\label{thm:RelHypEquivalent}
The six definitions of relative hyperbolicity given above are equivalent.
In addition, condition \textup{(RH-5)} does not depend on the choice
of relative generating set.
\end{thm}

Our strategy for proving the theorem is to explain the following implications:
\[
\xymatrix{
   \text{(RH-1)} \ar@{<=>}[r] & \text{(RH-2)} \ar@{<=>}[r] &
      \text{(RH-3)} \ar@{=>}[r] & \text{(RH-4)} \ar@{=>}[d] \\
           &   \text{(RH-6)} \ar@{=>}[ur] & \text{(RH-5w)} \ar@{=>}[l] \ar@{=>}[ur]
           &   \text{(RH-5s)} \ar@{=>}[l]
}
\]
where (RH-5s) denotes the strong condition that (RH-5) holds
for every finite relative generating set, and (RH-5w) denotes the weak condition
that it holds for some particular finite relative generating set.
We have included the implication (RH-5w) $\Longrightarrow$~(RH-4)
in order to give a more direct proof of the equivalence of the weak and strong
forms of (RH-5).

In fact most of the ingredients already exist in the literature.
Some are proved completely, and others are proved for the finitely generated
case but the proofs don't really use finite generation.
In several places we replace the Cayley graph for a finite generating set
with the (possibly disconnected) Cayley graph for a finite relative
generating set.
With this small change of perspective, the proofs in the literature
go through almost verbatim.

As explained in the previous section,
the most significant changes occur in our proof of
(RH-6) $\Longrightarrow$~(RH-3),
where we need to modify the Groves--Manning construction of combinatorial
horoballs to account for non--finitely generated groups.
Note that we also use combinatorial horoballs to prove
(RH-1) $\Longrightarrow$~(RH-2) in the elementary parabolic case.

The first implication we discuss is a consequence of the following theorem 
due to Yaman.

\begin{thm}[\cite{Yaman04}]
\label{thm:Yaman}
If $G$ acts as a nonelementary geometrically finite convergence group
on a metrizable compactum $M$, then there is a proper $\delta$--hyperbolic
space $X$ on which $G$ acts properly
such that $M$ is $G$--equivariantly homeomorphic to $\boundary X$
with its induced $G$--action.
\end{thm}

\begin{cor}\label{cor:RH12}
Definitions \textup{(RH-1)} and \textup{(RH-2)} are equivalent.
\end{cor}

\begin{proof}
The implication (RH-2)~$\Longrightarrow$~(RH-1) is obvious.
Now suppose $(G,\PP)$ satisfies (RH-1).
If the action of $G$ on $M$ is nonelementary, then we are done by
Theorem~\ref{thm:Yaman}.

If the action on $M$ is elementary, it is a simple matter to construct an
appropriate hyperbolic space $X$ as follows.
If $M$ is empty, then $G$ acts trivially on a point $X=\{a\}$.
If $M$ has one point then we choose a proper, left invariant metric $d$
on $G$, and let $G$ act properly on the combinatorial horoball
$X=\Set{H}(G,d)$.
If $M$ has two points, then $G$ is virtually cyclic and acts properly
on the line $X=\R$.
In this last case, arrange the action so that $g \in G$
preserves the orientation of $\R$
if and only if $g$ fixes both points of $M$.
In each case it is clear that the induced action on $\boundary X$
is equivalent to the given action on $M$, establishing (RH-2).
\end{proof}

The equivalence (RH-2)~$\Longleftrightarrow$~(RH-3) was discovered by Thurston
in the classical setting of $3$--dimensional Kleinian groups \cite{ThurstonNotes},
and has been extended to various manifold cases
by Apanasov and Bowditch \cite{Apanasov83,Bowditch93,Bowditch95}.
The following result due to Bowditch \cite{BowditchRelHyp}
extends this equivalence to actions on $\delta$--hyperbolic spaces.

\begin{thm}[\cite{BowditchRelHyp}]
\label{thm:GeomFiniteCuspUniform}
Let $G$ act properly on a proper $\delta$--hyperbolic space $X$.
Let $\PP$ be a finite family of subgroups of $G$.
Then the action of $(G,\PP)$ on $X$ is geometrically finite
if and only if it is cusp uniform.
\end{thm}

\begin{cor}
Definitions \textup{(RH-2)} and \textup{(RH-3)} are equivalent. \qed
\end{cor}

The following is one of the main theorems of \cite{BowditchRelHyp}.
We note that the result is stated with the extra assumption that the peripheral
subgroups are finitely generated, but this hypothesis plays no role in the
proof and is present only to make the result consistent with Bowditch's convention
that peripheral subgroups should be finitely generated.

\begin{thm}[\cite{BowditchRelHyp}]
Definition \textup{(RH-3)} implies \textup{(RH-4)}.
\end{thm}

As an aside, we remark that Bowditch also proves the converse
(RH-4) $\Longrightarrow$~(RH-3) for
finitely generated groups.  However, his proof uses the finite generation hypothesis
in an essential manner, and the proof does not extend to non--finitely
generated groups.  This implication is not necessary for our purposes.

Recall that in the finitely generated case (RH-5s) and (RH-5w)
are equivalent by a standard argument.
This standard argument does not extend in the obvious way
to the non--finitely generated case.

Bowditch claims without proof in \cite{BowditchRelHyp}
that (RH-4) and (RH-5) are equivalent
for finitely generated groups.
Dahmani provides a sketch in \cite{Dahmani03Combination} and a complete 
proof in an appendix to his thesis \cite{DahmaniThesis}.
Dahmani's proof extends with trivial modifications to the general case.
For the benefit of the reader, we sketch the outline of a proof based on 
Dahmani's arguments indicating places where modifications must be made.


\begin{proof}[Proof of \textup{(RH-4)}~$\Longrightarrow$~\textup{(RH-5s)}]
Suppose $G$ acts on a fine, $\delta$--hyperbolic graph $K$
with finite edge stabilizers and finitely many orbits of edges.
Choose edges representing the $G$--orbits, and let $\{v_1,\dots,v_\ell\}$
be the set of all vertices incident to any of these finitely many edges.
If $v_i$ and~$v_j$ lie in the same $G$--orbit, choose
$g_{ij} \in G$ so that $g_{ij}(v_i)=v_j$.
Then $G$ is generated by the stabilizers of the $v_i$ and
the finitely many elements $g_{ij}$
(see, for instance, Bridson--Haefliger \cite[Theorem~I.8.10]{BH99}).
In particular, $G$ is finitely generated with respect to the infinite vertex stabilizers.

Now let $\Set{S}$ be an arbitrary finite relative generating set for
$(G,\PP)$.
Observe that the coned-off Cayley graph $\hat\Gamma:=\hat\Gamma(G,\PP,\Set{S})$
has trivial edge stabilizers and finitely many orbits of edges.
The argument in \cite[Lemma~A.4]{DahmaniThesis}
produces a fine graph $K'$
with finite edge stabilizers and finitely many orbits of edges such that $K$
and $\hat\Gamma$
both embed equivariantly and simplicially into $K'$.
Any subgraph of a fine graph is fine, so $\hat\Gamma$ is fine.
Furthermore, a connected, equivariant subgraph of $K'$
is quasi-isometric to $K'$.
Since $K$ is $\delta$--hyperbolic, it follows that $\hat\Gamma$
is quasi-isometric to $K$,
so $\hat\Gamma$
is $\delta'$--hyperbolic for some $\delta'\ge 0$.
Now by \cite[Lemma~A.5]{DahmaniThesis}, the fineness of
the coned-off Cayley graph $\hat\Gamma$ implies that
$(G,\PP,\Set{S})$ has Bounded Coset Penetration.
(The proof of \cite[Lemma~A.5]{DahmaniThesis}
never uses the connectedness of $\Gamma$,
so it remains valid when $\Set{S}$ is a relative generating set,
as opposed to a generating set in the traditional sense.)
\end{proof}

\begin{proof}[Proof of \textup{(RH-5w)}~$\Longrightarrow$~\textup{(RH-4)}]
This implication follows directly from \cite[Proposition~A.1]{DahmaniThesis}
with no change, as Dahmani's proof does not use the finite generation
hypothesis.
\end{proof}

\begin{proof}[Proof of \textup{(RH-5w)}~$\Longrightarrow$~\textup{(RH-6)}]
This implication is proved by Osin in the finitely generated case
in an appendix to \cite{Osin06}.
However, he never uses the finite generation hypothesis in the proof.
Again, the finite generation hypothesis seems to be present only to 
make the statement consistent with Farb's original version of (RH-5) for finitely generated groups.
\end{proof}

We remark that
Osin also gives a substantially more involved proof of the converse
(RH-6)~$\Longrightarrow$~(RH-5) for finitely generated groups.
These arguments use finite generation extensively,
and it is not clear whether the same arguments go through in general.
Note that we do not require this implication.

\begin{proof}[Proof of \textup{(RH-6)}~$\Longrightarrow$~\textup{(RH-3)}]
Choose any finite relative generating set $\Set{S}$ for $(G,\PP)$
and choose a proper, left invariant metric $d_i$ for each subgroup
$P_i \in \PP$.
By Theorem~\ref{thm:AugmentedHyperbolic}, the augmented space $X$ corresponding 
to these data is connected and $\delta$--hyperbolic for some $\delta\ge 0$.
It is straightforward to verify that the combinatorial horoballs
of $X$ are also horoballs as defined in Section~\ref{sec:Horoballs}. 
Deleting these horoballs from $X$ gives the (possibly disconnected)
Cayley graph $\Cayley(G,\Set{S})$ for $G$,
on which $G$ acts with compact quotient.
Therefore $(G,\PP)$ is relatively hyperbolic in the sense of (RH-3).
\end{proof}

\section{Relatively quasiconvex subgroups: definitions}
\label{sec:QCDef}

Now that we have established the notion of relative hyperbolicity for countable 
groups, we turn our attention to the most natural class of subgroups, namely 
the relatively quasiconvex subgroups.
In this section we introduce these subgroups from several different points of 
view, corresponding to the various characterizations of relative hyperbolicity.
In each setting, we describe a geometrically natural family of
``relatively quasiconvex'' subgroups.
As before, our main goal is to show that the various definitions of 
relative quasiconvexity are equivalent.
As part of this equivalence, we will show that each of the conditions below
is an intrinsic property of the subgroup $H$ in the relatively hyperbolic group
$(G,\PP)$, and does not depend on any specific choices made throughout
the definition.

Throughout this section, we assume that $G$ is countable,
$\PP=\{P_1,\dots,P_n\}$ is a finite collection of subgroups,
and that $(G,\PP)$ is relatively hyperbolic.

\begin{defn}[Quasiconvex subspace]
Let $X$ be a geodesic metric space.
A subspace $Y \subseteq X$ is $\kappa$--quasiconvex
for some $\kappa > 0$ if every geodesic of $X$
connecting two points of $Y$ lies in the $\kappa$--neighborhood
of $Y$.
\end{defn}

\begin{defn}[QC-1]
A subgroup $H \le G$ is \emph{relatively quasiconvex}
if the following holds.
Let $M$ be some (any) compact, metrizable space on which $(G,\PP)$ acts
as a geometrically finite convergence group.
Then the induced convergence action of $H$ on the limit set
$\Lambda H \subseteq M$ is geometrically finite.
\end{defn}

The next two definitions of relative quasiconvexity
allow us to recognize relatively quasiconvex subgroups in the setting of
Gromov's original definition of relative hyperbolicity.

To express the first, we need the notion of the join of a set in the boundary.

\begin{defn}[Join]
Let $X$ be $\delta$--hyperbolic, and suppose $\Lambda \subseteq \boundary X$
is a subset with at least two points.
The \emph{join} of $\Lambda$, denoted $\join(\Lambda)$,
is the union of all geodesic lines in $X$ joining pairs of points in $\Lambda$.
\end{defn}

\begin{rem}\label{rem:join}
One should think of the join as a ``quasiconvex hull.''
In fact, it is easy to see that if $X$ is $\delta$--hyperbolic,
then $\join(\Lambda)$ is a $\kappa$--quasiconvex subspace for some
$\kappa = \kappa(\delta)$ (see for instance Gromov
\cite[Section~7.5.A]{Gromov87}).
By an elementary argument,
$\join(\Lambda)$ is quasi-isometric to a geodesic hyperbolic space
$Y$ whose boundary $\boundary Y$ is canonically homeomorphic to $\Lambda$.
Any horofunction on $X$ based at a point of $\Lambda$ restricts to a
horofunction on $Y$.
Consequently a horoball of $X$ centered in $\Lambda$ restricts to
a horoball of $Y$.
\end{rem}

\begin{defn}[QC-2]
A subgroup $H \le G$ is \emph{relatively quasiconvex} if the following holds.
Let $X$ be some (any) proper $\delta$--hyperbolic space on which $(G,\PP)$
has a cusp uniform action.  Then either $H$ is finite, $H$ is parabolic, 
or $H$ has a cusp uniform action
on a geodesic hyperbolic space $Y$ quasi-isometric
to the subspace $\join(\Lambda H) \subseteq X$.
%
\end{defn}

\begin{defn}[QC-3]
A subgroup $H \le G$ is \emph{relatively quasiconvex} if the following holds.
Let $(X,\rho)$ be some (any) proper $\delta$--hyperbolic space on which
$(G,\PP)$ has a cusp uniform action.
Let $X-U$ be some (any) truncated space for $G$ acting on $X$.
For some (any) basepoint $x \in X - U$
there is a constant $\mu\ge 0$ such that whenever $c$ is a geodesic in $X$
with endpoints in the orbit $Hx$, we have
\[
   c \cap (X-U) \subseteq \nbd{Hx}{\mu},
\]
where the neighborhood is taken with respect to the metric $\rho$ on $X$.
%
%
\end{defn}

The following construction was introduced by Farb \cite{Farb98}
in the context of manifolds with pinched negative curvature,
and further elaborated by Bowditch \cite{BowditchRelHyp}.
Farb credits the ``electric'' terminology to Thurston.

\begin{defn}[Electric space]
Suppose $(G,\PP)$ has a cusp uniform action on a $\delta$--hyperbolic space
$(X,\rho)$, and let $X-U$ be a truncated space for the action.
The \emph{electric pseudometric} $\hat{\rho}$ on $X$ is the path pseudometric
obtained by modifying $\rho$ so that it is identically zero on each
horoball of $U$
(essentially collapsing each horoball of $U$ to a point).
The \emph{electric space} associated to the pair $(X,U)$ is the
pseudometric space $(X,\hat{\rho})$.

A \emph{relative path} $(\gamma,\alpha)$ in $(X,U)$
is a path $\gamma$ together with a subset $\alpha$ consisting of a
disjoint union of finitely many paths $\alpha_1\dots,\alpha_n$
occurring in this order along $\gamma$
such that each $\alpha_i$ lies in some closed horoball $B_i$ of $U$.
We also assume that $B_i \ne B_{i+1}$.
The \emph{relative length} of $(\gamma,\alpha)$ is the length
in $X$ of $\gamma - \alpha$.
We write $\gamma = \beta_0 \cup \alpha_1 \cup \beta_1 \cup \cdots
\cup \alpha_n \cup \beta_n$, where the $\beta_i$ are the complementary paths
of $\gamma$. (Either of $\beta_0$ or $\beta_n$ could be a trivial path.)
A relative path $(\gamma,\alpha)$ is \emph{efficient} if $B_i \ne B_j$
for $i\ne j$, and is \emph{semipolygonal} if each $\alpha_i$
is a geodesic of $X$.

There are natural notions of a relative path $(\gamma,\alpha)$
being an \emph{electric geodesic} or an
\emph{electric $\epsilon$--quasigeodesic}
given by comparing the relative length of an arbitrary subsegment
with the electric distance between its endpoints in the obvious way.
%
%
\end{defn}

The following result, roughly speaking, says that an electric quasigeodesic
is also a quasigeodesic in $(X,d)$.

\begin{lem}[Lemma~7.3, \cite{BowditchRelHyp}]
\label{lem:LiftIsQuasigeodesic}
Given constants $\delta,\epsilon>0$, there are constants
$r=r(\delta)$ and $\epsilon'=\epsilon'(\delta,\epsilon)$
so that the following holds.
Suppose $(G,\PP)$ has a cusp uniform action
on a $\delta$--hyperbolic space $(X,\rho)$.
Let $X-U$ be a truncated space for the action
such that any two horoballs of $U$ are separated in $X$
by a $\rho$--distance at least $r$, and let $\hat{\rho}$ be the corresponding
electric metric.
If $(\gamma,\alpha)$ is an efficient, semipolygonal relative path in $(X,U)$
and also an electric $\epsilon$--quasigeodesic
then $\gamma$ is an $\epsilon'$--quasigeodesic in $X$.
\end{lem}

Roughly speaking the following definition says that a subgroup is
relatively quasiconvex if its orbits are quasiconvex with respect to electric
geodesics.

\begin{defn}[QC-4]
A subgroup $H \le G$ is \emph{relatively quasiconvex}
provided that the following holds.
Let $(G,\PP)$ have a cusp uniform action on
some (any) $\delta$--hyperbolic space $(X,d)$,
and choose some (any) truncated space $X-U$ for the action.
Suppose each pair of horoballs of $U$ is separated by at least a distance
$r$, where $r=r(\delta)$
is the constant given by Lemma~\ref{lem:LiftIsQuasigeodesic}.

For some (each) basepoint $x \in X-U$ 
there is a constant $\kappa >0$ so that for every
efficient, semipolygonal relative path $(\gamma,\alpha)$ that is also an
electric geodesic connecting points of $Hx$,
the subset $\gamma - \alpha$
lies in the $\kappa$--neighborhood of $Hx$ in $(X,\rho)$.
\end{defn}

Observe that a similar condition will hold for electric quasigeodesics
connecting points of the orbit $Hx$, since Lemma~\ref{lem:LiftIsQuasigeodesic}
implies that electric quasigeodesics with common endpoints
track Hausdorff close in $X$.

The next definition is quite close to the one introduced by Osin for subgroups
of a finitely generated relatively hyperbolic group.  The only difference between Osin's definition and the one below is that
we use a proper, left invariant metric $d$ on $G$ where Osin used
the word metric for a finite generating set.
It is not clear whether one can use the word metric for a
finite relative generating set in place of the proper, left invariant metric.

\begin{defn}[QC-5]
A subgroup $H \le G$ is \emph{relatively quasiconvex} if the following holds.
Let $\Set{S}$ be some (any) finite relative generating set for $(G,\PP)$,
and let $\Set{P}$ be the union of all $P_i \in \PP$.
Consider the Cayley graph $\bar\Gamma=\Cayley(G,\Set{S} \cup \Set{P})$
with all edges of length one.
Let $d$ be some (any) proper, left invariant metric on~$G$.
Then there is a constant $\kappa=\kappa(\Set{S},d)$ such that
for each geodesic $\bar{c}$ in $\bar\Gamma$
connecting two points of $H$,
every vertex of $\bar{c}$ lies within a $d$--distance $\kappa$ of $H$.
\end{defn}

We will see that the preceding definition is independent of the choice of
finite relative generating set $\Set{S}$ and the choice of proper metric $d$.
In particular, if $\Set{S}$ is a finite generating set for $G$
in the traditional sense, then we can choose $d$ to be the word metric
for $\Set{S}$.
Thus in the finitely generated case,
(QC-5) is equivalent to the definition given by Osin in \cite{Osin06}.

%
%

\section{Relative quasiconvexity: Equivalence of definitions}
\label{sec:QCEquivalence}

We now establish the equivalence of the various definitions
of relative quasiconvexity
introduced in the previous section.  Throughout this section
$(G,\PP)$ is a relatively hyperbolic group and $H$ is any subgroup of $G$.

\begin{prop}
\label{prop:QC12}
Definitions \textup{(QC-1)} and \textup{(QC-2)} are each well-defined
and are equivalent.
\end{prop}

The proof uses the following result due to Bowditch.

\begin{thm}[Theorem~9.4, \cite{BowditchRelHyp}]
\label{thm:BoundaryWellDefined}
Suppose $(G,\PP)$ has cusp uniform actions on spaces $X$ and $X'$.
Then $\boundary X$ and $\boundary X'$ are $G$--equivariantly homeomorphic.
\end{thm}

\begin{proof}[Proof of Proposition~\ref{prop:QC12}]
By Corollary~\ref{cor:RH12}, if $(G,\PP)$ has a geometrically finite
convergence action on a compact, metrizable space $M$,
then $(G,\PP)$ acts geometrically finitely
on a proper $\delta$--hyperbolic space $X$
and $M$ is $G$--equivariantly homeomorphic to the boundary of $X$.
It follows from Theorem~\ref{thm:BoundaryWellDefined} that condition (QC-1)
does not depend on the choice of compactum $M$
or the choice of geometrically finite action of $(G,\PP)$ on $M$.

Suppose $(G,\PP)$ has a cusp uniform action on a proper $\delta$--hyperbolic
space $X$.  Then the induced action on $M=\boundary X$ is
a geometrically finite convergence action.
We will show that (QC-1) holds for the action of $G$ on $M$
if and only if (QC-2) holds for the action of $G$ on $X$.
By the preceding paragraph, it follows immediately that
(QC-2) is independent of the choice of $X$.

Recall our convention that the unique action of a finite group $H$
on the empty set is a convergence group action.
Therefore each finite subgroup $H\le G$ satisfies both (QC-1) and
(QC-2).
Similarly, the unique action of a countable group $H$ on a single point
is a convergence group action.
If $H$ is parabolic then $\Lambda H$ is a single point,
and hence $H$ satisfies both (QC-1) and (QC-2).

Now suppose $H$ is neither finite nor parabolic.
Then $\Lambda H \subseteq \boundary X$ contains at least two points,
and $\join(\Lambda H)$ is a nonempty quasiconvex subspace of $X$
invariant under the action of $H$.
By Remark~\ref{rem:join}, the subspace $\join(\Lambda H)$ with the subspace 
metric is quasi-isometric to a $\delta'$--hyperbolic geodesic space $Y$
for some $\delta'\ge 0$,
and $\boundary Y$ is $H$--equivariantly homeomorphic
to $\Lambda H \subseteq \boundary X$.

Applying Theorem~\ref{thm:GeomFiniteCuspUniform} to $Y$,
we see that (QC-1) holds for $\boundary Y = \Lambda H$
if and only if (QC-2) holds for $Y$.
\end{proof}

\begin{lem}\label{lem:TwoGeodesicsTruncated}
Let $(G,\PP)$ have a cusp uniform action on
a $\delta$--hyperbolic space $(X,\rho)$
with associated truncated space $X - U$.
For each $N$ there is a constant $M_1$ so that the following holds.
Choose two $\rho$--geodesics $c,c'$ such that
\[
   \max \{ \rho(c_{-},c'_{-}), \rho(c_{+},c'_{+}) \} < N.
\]
If $c \cap (X-U)$ and $c' \cap (X-U)$ are both nonempty then
the Hausdorff $\rho$--distance between $c \cap (X-U)$ and $c' \cap (X-U)$
is less than $M_1$.
\end{lem}

\begin{proof}
The geodesics $c$ and $c'$ are opposite sides of a $2\delta$--thin
geodesic quadrilateral.
It follows that the Hausdorff distance in $X$ between $c$ and $c'$ is at most
$N' := N + 2\delta$.
Choose a point $p \in c \cap (X-U)$.
Then $p$ is within a distance $N'$ of a point $q \in c'$.
Observe that $q$ also lies in $\nbd{X-U}{N'}$.
By Lemma~\ref{lem:NearHoroballComplement} there is a segment of $c'$ from 
$q$ to $X-U$ of length at most $M_0$, where $M_0$ depends only on
$X-U$ and $N'$.
Therefore $p$ lies within a distance $N' + M$ of a point $r \in c'\cap(X-U)$.
Interchanging the roles of $c$ and $c'$ completes the proof.
\end{proof}

\begin{lem}\label{lem:TwoTruncatedSpaces}
Let $(G,\PP)$ have a cusp uniform action
on a $\delta$--hyperbolic space $(X,\rho)$.
Let $X-U$ and $X-U'$ be two corresponding truncated spaces.
Then there is a constant $M_2$ such that for each $\rho$--geodesic $c$ in $X$,
if $c \cap (X-U)$ and $c \cap (X-U')$ are both nonempty then
the Hausdorff $\rho$--distance between $c \cap (X-U)$ and $c \cap (X-U')$
is at most $M_2$.
\end{lem}

\begin{proof}
Recall that the horoballs of $U$ and $U'$ lie in finitely many $G$--orbits,
corresponding to the finitely many peripheral subgroups of $\PP$.
By shrinking the horoballs of $U$ and $U'$ equivariantly,
we can obtain a third truncated space $X-V$ such that $V \subseteq U \cap U'$.
If $c$ has nonempty intersection with each of the given truncated spaces
then it also has nonempty intersection with $X-V$.
Thus it suffices to prove the lemma in the special case when $U' \subseteq U$.
Moreover, it is enough to consider the further special case 
that $U$ and $U'$ differ on only one orbit of horoballs,
since this case can be iterated finitely many times to obtain the desired
result.

Suppose $U$ and $U'$ agree on all but one orbit, and this exceptional orbit
is represented by horoballs $B$ and $B'$ of $U$ and $U'$ respectively,
such that $B' \subset B$.
Then $X-U \subset X-U'$.
Clearly we have
\[
   c \cap (X-U) \subseteq c \cap (X-U').
\]
On the other hand, any subsegment $\bar{c}$ of $c$
that lies in $X-U'$ but not in $X-U$
is a translate of a geodesic in $B-B'$.
By Lemma~\ref{lem:BetweenHoroballs}, such a segment has length at most
$M_1=M_1(B,B')$.
Thus
\[
   c \cap (X-U') \subseteq \bignbd{ c \cap (X-U) }{M_0},
\]
completing the proof.
\end{proof}

\begin{prop}\label{prop:QC3WellDefined}
Suppose $(G,\PP)$ has a cusp uniform action on $(X,\rho)$, and $H \le G$.
If $H$ satisfies \textup{(QC-3)}
with respect to one choice of truncated space $X-U$
and basepoint $x \in X-U$, then $H$ satisfies \textup{(QC-3)}
with respect to any other
truncated space $X-U'$ and basepoint $x'$.
\end{prop}

\begin{proof}
Let us first consider the effect of changing the basepoint $x \in X-U$ to another point $x' \in X-U$ in the same truncated space.
Choose $h_0,h_1 \in H$, a $\rho$--geodesic $c$ from $h_0(x)$ to $h_1(x)$,
and a $\rho$--geodesic $c'$ from $h_0(x')$ to $h_1(x')$.
Then
\[
   \rho(c_{-},c'_{-}) = \rho(c_{+},c'_{+}) = \rho(x,x').
\]
By Lemma~\ref{lem:TwoGeodesicsTruncated}, the Hausdorff distance between
$c \cap (X-U)$ and $c' \cap (X-U)$ is at most $M_1$
for some constant $M_1$ that depends on $x$ and $x'$ but not on the choice
of $h_0$ and $h_1$.
If $c \cap (X-U)$ lies in $\nbd{Hx}{\kappa}$
then $c' \cap (X-U)$ lies in $\nbd{Hx'}{\kappa + M_1}$.
Therefore if we fix $X-U$, then (QC-3) does not depend on the choice of 
basepoint $x \in X-U$.

Now suppose $X-U$ and $X-U'$ are two truncated spaces
corresponding to the action of $G$ on $X$.
Choose basepoints $x \in X-U$ and $x' \in X-U'$.
Shrinking horoballs, we can obtain a third truncated space $X-V$ with
$V \subseteq U \cap U'$.
In particular, note that $x$ and $x'$ both lie in $X-V$.
If $H$ satisfies (QC-3) with respect to $X-U$ with basepoint $x$,
then by Lemma~\ref{lem:TwoTruncatedSpaces}
it also satisfies (QC-3) with respect to $X-V$ with basepoint $x$.
By the preceding paragraph, $H$ also satisfies
(QC-3) with respect to $X-V$ and basepoint $x'$.
Finally, another application of Lemma~\ref{lem:TwoTruncatedSpaces}
shows that $H$ satisfies (QC-3) with respect to $X-U'$ and basepoint $x'$.

Thus (QC-3) is independent of the choice of truncated space $X-U$
and independent of the choice of basepoint $x \in X-U$.
\end{proof}

\begin{prop}\label{prop:QC23}
Definition \textup{(QC-3)} is well-defined and equivalent to
Definition \textup{(QC-2)}.
\end{prop}

\begin{proof}
In Proposition~\ref{prop:QC3WellDefined} we showed that 
for each fixed cusp uniform action of $(G,\PP)$ on a space $(X,\rho)$,
whether a subgroup $H$ satisfies
(QC-3) does not depend on the choice of truncated space and basepoint.
In Proposition~\ref{prop:QC12} 
we showed that (QC-2) does not depend on the choice of space $X$
or the choice of cusp uniform action of $(G,\PP)$ on $X$.
Thus, in order to show that (QC-3) is also independent of the space $X$
and cusp uniform action,
it suffices to show that (QC-2) holds for a particular action of $G$ on $X$
if and only if (QC-3) holds for the same action with respect to some choice
of truncated space $X-U$ and basepoint $x \in X-U$.

(QC-2)~$\Longrightarrow$~(QC-3):
Choose any truncated space $X-U$ for the action of $G$ on $X$.
If $H$ is finite, then (QC-3) is immediate since the orbit $Hx$ is bounded.
If $H$ is parabolic with parabolic fixed point $p \in \boundary X$,
and $x \in X-U$ is any basepoint,
then there is a horofunction based at $p$ that is identically zero 
on the orbit $Hx$.
Let $B$ be the horoball of $U$ centered at $p$.
By Lemma~\ref{lem:GeodesicInHoroball}, there is a constant $M_2$
such that for each geodesic $c=[a,b]$ connecting two points of $Hx$,
we have
\[
   c \cap (X-U) \subseteq \bignbd{\{a,b\}}{M_2}.
\]
Thus (QC-3) holds for $H$.

Now suppose $\Lambda H$ contains at least two points.
Then
$\join(\Lambda H)$ is nonempty and $\kappa$--quasiconvex for some $\kappa>0$.
Suppose the action of $H$ on $Y$ is cusp uniform,
where $Y$ is quasi-isometric to $\join(\Lambda H)$.
Under this quasi-isometry, every horoball of $Y$ pulls back to a subset
of $\join(\Lambda H)$ of the form $B \cap \join(\Lambda H)$ for some horoball
$B$ of $X$.
Therefore, the action of $H$ on $\join(\Lambda H)$ is ``cusp uniform''
in the sense that
$X$ contains a union of horoballs $U'$ centered at the parabolic
points of $H$ such that the horoballs lie in finitely many $H$--orbits
and such that $H$ acts cocompactly on the 
``truncated space'' $\join(\Lambda H) - U'$.

Shrink the horoballs of $U'$ $H$--equivariantly so that
$\rho(U',X-U) > \kappa$.
Then $H$ still acts cocompactly on $\join(\Lambda H) - U'$.
Let $\nu < \infty$ be the $\rho$--diameter of a compact set $K\subseteq X$
whose $H$--translates cover $\join(\Lambda H) - U'$.
Choose a basepoint $x \in \join(\Lambda H) - U'$,
and let $c$ be a geodesic in $X$ connecting two points of $Hx$.
Since the endpoints of $c$ lie in $\join(\Lambda H)$,
the $\kappa$--quasiconvexity of $\join(\Lambda H)$ implies that
each $p \in c \cap (X-U)$ is within a distance $\kappa$
of a point $q \in \join(\Lambda H)$.
Observe that $q$ is also in $\nbd{X-U}{\kappa}$.
By our choice of $U'$, it follows that both $x$ and $q$ lie in
$\join(\Lambda H) - U'$.
Therefore, for some $h \in H$, we have $\rho(hx,q) < \nu$,
establishing (QC-3).

(QC-3)~$\Longrightarrow$~(QC-2):
Suppose $(G,\PP)$ has a cusp uniform action on $(X,\rho)$,
and $H \le G$ is infinite and not parabolic.
Then $\join(\Lambda H)$
is nonempty and $\kappa$--quasiconvex for some $\kappa>0$,
and quasi-isometric to a hyperbolic geodesic space $Y$.
In order to establish (QC-2), it suffices to show that the action
of $H$ on $Y':=\join(\Lambda H)$ is cusp uniform in the above sense.

Choose a truncated space $X-U$ for the cusp uniform action of $G$ on $X$
and a basepoint $x \in Y'-U$
such that $H$ satisfies (QC-3) with respect to $X-U$ and $x$
using quasiconvexity constant $\mu$.

Then any geodesic $c$ connecting two points of $Hx$
satisfies
\[
   c \cap (X-U) \subseteq \nbd{Hx}{\mu}
\]
If $\bar{c}$ is a bi-infinite geodesic in $X$ obtained as a pointwise limit of 
such segments $c$, then $\bar{c}$ connects two points of $\Lambda H$
and satisfies
\[
   \bar{c} \cap (X-U) \subseteq \nbd{Hx}{\mu}.
\]
Since $X$ is $\delta$--hyperbolic, and each point of $\Lambda H$ is a limit
point of $Hx$, every geodesic connecting two points of $\Lambda H$
lies in the $2\delta$--neighborhood of such a geodesic $\bar{c}$.
(This is because any two bi-infinite geodesics with the same endpoints at infinity in a $\delta$--hyperbolic space have Hausdorff distance
at most $2\delta$.)
Therefore
\[
   Y'-U \subseteq \nbd{Hx}{2\delta + \mu},
\]
proving that $H$ acts cocompactly on $Y' - U$.

In order to establish (QC-2), we must show that $H$ acts cocompactly on
$Y' - U'$, where $U'$ is an $H$--equivariant family of
horoballs centered only at the parabolic points of $H$.
It suffices to show that the horoballs $B$ of $U$ that meet $Y'$
and are not based in $\Lambda H$ have uniformly bounded intersections
with $Y'$.  Since then $Y' - U$
is quasidense in $Y'-U'$.

If a horoball $B$ of $U-U'$ meets $Y'$, then it also meets $Y' - U$,
which lies in $\nbd{Hx}{2\delta + \mu}$.
Hence, $\rho(B,Hx)$ is at most $2\delta + \mu$.
Then the translate $hB$ of $B$ by some element $h \in H$ intersects
$\ball{x}{2\delta + \mu}$.
Since only finitely many horoballs of $U$ meet any metric ball in $X$,
the horoballs of $U$ meeting $Y'$ must lie in only finitely many $H$--orbits.
If such a horoball is not based in $\Lambda H$, then its intersection
with $Y'$ is bounded.  Combining these bounds for the finitely many $H$--orbits
gives a uniform bound on the diameter of all such intersections,
which completes the proof.
\end{proof}

\begin{prop}\label{prop:QC34}
Definition \textup{(QC-4)} is well-defined and equivalent to \textup{(QC-3)}.
\end{prop}

\begin{proof}
Recall that we have shown that if a subgroup $H\le G$ satisfies (QC-3)
with respect to a cusp uniform action, truncated space, and basepoint,
then it satisfies (QC-3) with respect to any other such choices.
Choose a proper $\delta$--hyperbolic space $(X,\rho)$, a cusp uniform 
action of $(G,\PP)$ on $X$,
a truncated space $X-U$, and a basepoint $x \in X-U$.
Suppose further that the horoballs of $U$ are pairwise separated by
a $\rho$--distance at least $r$, where $r=r(\delta)$ is the constant provided
by Lemma~\ref{lem:LiftIsQuasigeodesic}.
We will show that $H$ satisfies (QC-3) with respect to these choices
if and only if it satisfies (QC-4).

More specifically, we will show that there is a constant $L$ such that
the following holds.
Let $c$ be a $\rho$--geodesic and $(\gamma,\alpha)$ a semipolygonal
relative path that is also an electric geodesic.
If $c$ and $\gamma$ have the same endpoints in $Hx$, then the
Hausdorff $\rho$--distance between $c \cap (X-U)$ and $\gamma - \alpha$
is at most $L$.

By Lemma~\ref{lem:LiftIsQuasigeodesic}, the path $\gamma$ is an
$\epsilon$--quasigeodesic in $(X,\rho)$ for some uniform constant $\epsilon$.
By the Morse Lemma, the Hausdorff $\rho$--distance between $c$ and $\gamma$
is at most $\eta = \eta(\delta,\epsilon)$.
Since $(\gamma,\alpha)$ is an electric geodesic with respect to $(X,U)$,
the subset $\gamma - \alpha$ must lie in $X-U$.
Therefore any point $p \in \gamma-\alpha$ is within a $\rho$--distance
$\eta$ of a point $q \in c \cap \nbd{X-U}{\eta}$.
By Lemma~\ref{lem:NearHoroballComplement}
the point $q$ is within a $\rho$--distance
$M_0$ of a point $q' \in c \cap (X-U)$,
where $M_0$ depends only on $\eta$ and $U$.
(Here we are using that $\nbd{X-U}{\eta}$
differs from $X-U$ on only finitely many orbits of horoballs.)
A similar argument (using a quasigeodesic variation of
Lemma~\ref{lem:NearHoroballComplement})
bounds the distance from an arbitrary point
of $\gamma-\alpha$ to $c \cap (X-U)$.
Thus we have an upper bound on the Hausdorff $\rho$--distance between
$\gamma-\alpha$ and $c \cap (X-U)$.
It is now clear that if either set lies near $Hx$, then so does the other.
\end{proof}

It is a well-known result, first observed by Efromovich \cite{Efromovich53}
that if $X$ is a connected length space,
and $G$ is a group acting metrically properly, coboundedly,
and isometrically on $X$,
then $G$ is finitely generated and quasi-isometric to $X$.
Recall that a length space is connected if and only if all distances
are finite.
The following result dealing with actions on arbitrary
length spaces is an easy corollary to Efromovich's Theorem.

\begin{prop}[Disconnected Efromovich's Theorem]
\label{prop:EfromovichDisconnected}
Let $X$ be a length space \textup{(}possibly disconnected\textup{)}.
Let $G$ be a countable group with proper, left invariant metric $d$
acting metrically properly, coboundedly, and isometrically on $X$.
Let $Y$ be a component of $X$.
Then the $G$--translates of $Y$ cover $X$,
the stabilizer $H$ of $Y$ has a finite generating set $\Set{S}$,
and for each basepoint $x_0 \in X$, the map $g \mapsto g(x_0)$
induces a quasi-isometry from the \textup{(}possibly
disconnected\textup{)} Cayley graph $\Cayley(G,\Set{S})$ to $X$.
\end{prop}

\begin{proof}
Since any two components of $X$ are separated by an infinite distance,
it is clear that $H$ acts coboundedly on $Y$.
Applying Efromovich's theorem to the action of $H$ on the connected
length space $Y$ gives a finite generating set $\Set{S}$ for $H$ and shows
that $\Cayley(H,\Set{S})$ is quasi-isometric to $Y$.
Note that if $H \ne G$, then $\Cayley(G,\Set{S})$ and $X$
are both disconnected with all components separated by infinite distances.
Since $G$ acts on $X$ with a connected quotient, the translates of $Y$
cover $X$.
In other words, the path components of $X$ are precisely the translates
of $Y$, which are in one-to-one
correspondence with the left cosets $gH$ of $H$ in $G$,
which in turn correspond to the path components in the Cayley graph.
\end{proof}

\begin{prop}
\label{prop:ElectricQIRelCayley}
Suppose $(G,\PP)$ is relatively hyperbolic
with finite relative generating set $\Set{S}$.
Let $(G,\PP)$ have a cusp uniform action on $(X,\rho)$
with associated truncated space $X-U$.
Suppose the horoballs of $U$ are pairwise separated by some minimum
positive distance $r$.

For each basepoint $x \in X-U$ the orbit map $g \mapsto g(x)$
extends to a quasi-isometry
\[
   f\colon \Cayley(G,\Set{S} \cup \Set{P}) \to (X,\hat\rho),
\]
where $\hat\rho$ denotes the electric pseudometric for $(X,U)$.
Furthermore, $f$ can be chosen so that each geodesic edge-path
in the Cayley graph maps to an efficient semi-polygonal relative
path in $(X,U)$.
\end{prop}

\begin{proof}
Choose a basepoint $x \in X-U$, and let $Z$ be the rectifiable-path
component of $X-U$ containing $x$.  Then $Z$ is connected when endowed
with the induced length metric.
Let $Y$ be the union of all $G$--translates of $Z$.
Since $G$ acts cocompactly on $(X-U,\rho)$ and $Y$ is a nonempty,
$G$--equivariant subspace of $X-U$, it follows that $Y$ is
quasi-dense in $(X-U,\rho)$.
Therefore, $G$ acts on $Y$ with a rectifiable-path connected quotient.
Let $\bar\rho$ denote the length metric on $Y$ induced by $\rho$.
Then $G$ acts coboundedly, metrically properly, and isometrically on
$(Y,\bar\rho)$.
By Proposition~\ref{prop:EfromovichDisconnected}, the stabilizer
$H$ of $Z$ is generated by a finite set $\Set{T}$, and
the orbit map $g \mapsto g(x)$ induces a quasi-isometry
$\Cayley(G,\Set{T}) \to (Y,\bar\rho)$.
Under this quasi-isometry, the set of left cosets $gP$ for all $g \in G$
and all $P \in \PP$ corresponds with
the set of horoballs in $U$.
The electric space $(X,\hat{\rho})$ is quasi-isometric
to its quasi-dense subspace $(Y,\hat{\rho})$,
which in turn is clearly quasi-isometric to
$\Cayley(G,\Set{T} \cup \Set{P})$.
In particular, since the electric space is connected, this Cayley graph
is also connected, establishing that $\Set{T}$
is a relative generating set for $(G,\PP)$.
By an elementary argument (see, for instance, Osin
\cite[Proposition~2.8]{Osin06}),
the identity $G \to G$ induces a quasi-isometry
\[
   \Cayley(G,\Set{S} \cup \Set{P})
   \to \Cayley(G,\Set{T} \cup \Set{P})
\]
for any other relative generating set $\Set{S}$.
Composing these quasi-isometries gives a quasi-isometry
\[
   f\colon\Cayley(G,\Set{S} \cup \Set{P})
    \to (X,\hat\rho)
\]
mapping $g \mapsto g(x)$ for each $g \in G$.

Since $G$ is quasi-dense in the Cayley graph,
we can modify $f$ so that it
sends the edges of the Cayley graph to any relative paths we like,
provided that the relative lengths of these paths are uniformly bounded. 
For each generator $s \in \Set{S}$, choose
a $\rho$--geodesic $\beta_s$ from $x$ to $s(x)$, considered as a relative path
$(\gamma,\alpha)$ with $\alpha$ empty.
Let $f$ map each Cayley graph edge labeled by $s$
to the appropriate translate of $\beta$.
For each $P \in \mathcal{P}$,
choose a $\rho$--geodesic $\beta_P$ from $x$ to the horoball $B$
stabilized by $P$.
For each generator labeled by $p \in P$, let $(\gamma_p,\alpha_p)$
be the relative path with $\gamma_p := \beta_P \cup \alpha_p \cup
\overline{(p \of \beta_P)}$, where $\alpha_p$ is a $\rho$--geodesic
path in $B$ from the terminal point $b$ of $\beta_P$ to the point $p(b)$,
and where the bar indicates following a path
in the reverse direction.
Let $f$ map each Cayley graph labeled by $p$ to the appropriate translate
of $(\gamma_p,\alpha_p)$.

Observe that each edge path in the Cayley graph is mapped by $f$
to a semipolygonal relative path. Furthermore, if the image $(\gamma,\alpha)$
is not efficient, and $\alpha$ is a disjoint union of paths
$\alpha_1,\dots,\alpha_n$, then for some $i<j$
the paths $\alpha_i$ and $\alpha_j$ lie in the same horoball $B$.
Let $e_i$ and $e_j$ be the corresponding edges in the Cayley graph.
It follows that the initial point $v$ of $e_i$ and the terminal point $w$
of $e_j$ lie in the same left coset $gP$ of some $P \in \PP$.
In particular, in the Cayley graph, the distance from $v$ to $w$
is at most $1$, so the given edge path cannot be geodesic.
\end{proof}

\begin{prop}\label{prop:QC45}
Definition \textup{(QC-5)} is well-defined and equivalent to \textup{(QC-4)}.
\end{prop}

\begin{proof}
Fix a finite relative generating set $\Set{S}$ for $(G,\PP)$, and suppose
$H$ satisfies (QC-5)
with respect to $\Set{S}$ and some proper left-invarant metric $d$ with
quasiconvexity constant $\kappa=\kappa(\Set{S},d)$.
We first show that (QC-5) continues to hold if we replace $d$ with another
proper left-invariant metric $d'$.
Let $\Set{B}$ be the finite ball of radius $\kappa$ centered at the identity
in the metric space $(G,d)$.  Let
\[
   \kappa' := \max \set{d'(1,g)}{g \in \Set{B}} + 1 < \infty.
\]
We will see that (QC-5) holds for $\Set{S}$ and $d'$ with
$\kappa(\Set{S},d') := \kappa'$.
Let $\bar{c}$ be any geodesic in $\Cayley(G,\Set{S} \cup \Set{P})$
connecting two points of $H$,
and let $v \in G$ be a vertex of $c$.
Then there is an element $w \in H$ such that $d(v,w) < \kappa$.
Since $d$ is left-invariant, $d(1,v^{-1}w) < \kappa$ so that
$v^{-1}w \in \Set{B}$.
But then $d'(v,w) = d'(1,v^{-1}w) < \kappa'$ as desired,
establishing (QC-5) for $\Set{S}$ and $d'$.
Note that the above argument holds even if $d$ or $d'$ is a
pseudometric (rather than a metric) provided that both are proper
and left-invariant.

Now let us consider the equivalence of (QC-4) and (QC-5).
Fix a finite relative generating set $\Set{S}$ for $(G,\PP)$ and
a cusp uniform action of $(G,\PP)$
on a proper $\delta$--hyperbolic space $(X,\rho)$
with an associated truncated space $X-U$.
Suppose further that the horoballs of $U$ are pairwise separated
by at least a distance $r$, where $r=r(\delta)$ is the constant given by
Lemma~\ref{lem:LiftIsQuasigeodesic}.
Fix a basepoint $x \in X-U$.
Define a proper left-invariant pseudometric $d_G$ on $G$
using distances in $X$ between orbit points;
that is, we set
\[
   d_G(g_1,g_2) := \rho\bigl(g_1(x),g_2(x) \bigr).
\]
To prove the proposition, it suffices to show that
(QC-5) holds for $\Set{S}$ and $d_G$ if and only if
(QC-4) holds for $X$, $U$, and $x$.
Since $\Set{S}$ is arbitrary,
it therefore follows that (QC-5) is independent of the choice of $\Set{S}$.

(QC-4)~$\Longrightarrow$~(QC-5):
%
By Proposition~\ref{prop:ElectricQIRelCayley}, there is an
$\epsilon$--quasi-isometry
\[
   f\colon\Cayley(G,\Set{S} \cup \Set{P}) \to (X,\hat\rho)
\]
for some $\epsilon>0$,
such that $f$ maps $g\in G$ to $g(x)$
and such that each geodesic $c$ in $\Cayley(G,\Set{S} \cup \Set{P})$
with endpoints in $H$
maps to an efficient, semipolygonal relative $\epsilon$--quasigeodesic
$(\gamma,\alpha)$
in $(X,U')$ such that $\gamma$ has endpoints in $Hx$.
Let $(\gamma',\alpha')$ be an efficient, semipolygonal relative geodesic
such that $\gamma'$ has the same endpoints as $\gamma$.
By Lemma~\ref{lem:LiftIsQuasigeodesic}, $\gamma$ and $\gamma'$ are
$\epsilon'$--quasigeodesics in $(X,\rho)$
for some $\epsilon'$ depending on $\epsilon$.
In particular, by the Morse Lemma
the Hausdorff $\rho$--distance between $\gamma$ and $\gamma'$
is at most $L=L(\delta,\epsilon')$.
As in the proof of Proposition~\ref{prop:QC34}
there is a uniform upper bound $L'$ on the Hausdorff $\rho$--distance
between $\gamma - \alpha$ and $\gamma' - \alpha'$.

Choose an arbitrary vertex $g$ of $c$.  Being a vertex of the Cayley
graph, $g$ is an element of $G$.  Thus $f$ maps $g$ to the orbit point
$g(x) \in \gamma - \alpha$, which is within a $\rho$--distance $L'$ of
some $y \in \gamma' - \alpha'$.
By (QC-4), the point $y$ is within a $\rho$--distance $\kappa$ of $Hx$.
Applying the triangle inequality shows that $\rho(g(x),Hx) < L + \kappa$.
By our definition of $d_G$, we also have $d_G(g,H) < L+\kappa$,
establishing (QC-5).

(QC-5)~$\Longrightarrow$~(QC-4):
Let $(\gamma',\alpha')$ be an efficient, semipolygonal relative geodesic
in $(X,U)$ such that $\gamma'$ has endpoints $hx$ and $h'x$.
Let $c$ be a geodesic in $\Cayley(G,\Set{S} \cup \Set{P})$
with endpoints $h$ and $h'$.
As above, $f$ maps $c$ to an efficient, semipolygonal relative
$\epsilon$--quasigeodesic $(\gamma,\alpha)$, and the Hausdorff
$\rho$--distance between $\gamma-\alpha$ and $\gamma'-\alpha'$ is
bounded above by a constant $L'$ that does not depend on the choice of
$h,h' \in H$.
Each point $y \in \gamma' - \alpha'$ is within a $\rho$--distance $L'$
of a point $z \in \gamma - \alpha$, which is within a $\rho$--distance
$\epsilon$ of $gx$ for some vertex $g$ of $c$.
By (QC-5), we have $d_G(g,H) <\kappa$.
Thus by the definition of $d_G$, we have $\rho\bigl( g(x),Hx \bigr) <\kappa$
as well.
By the triangle inequality it follows that
\[
   \rho(y,Hx) < L' + \epsilon + \kappa,
\]
establishing (QC-4).
\end{proof}

\section{Relative quasiconvexity in the word metric}
\label{sec:QCWordMetric}

In this section we characterize relatively
quasiconvex subgroups of finitely generated
relatively hyperbolic groups in terms of the geometry of the word metric.
We emphasize that throughout this section $(G,\PP)$ always denotes a finitely 
generated relatively hyperbolic group, and $\Set{S}$ is always a finite 
generating set for $G$ \emph{in the traditional sense}.
As above, $\Set{P}$ denotes the union of the peripheral subgroups $P\in\PP$.

If $c$ is a path in $\Cayley(G,\Set{S} \cup \Set{P})$,
a \emph{lift} of $c$ is a path formed from $c$ by replacing each edge
of $c$ labelled by an element of $\Set{P}$ with a geodesic in
$\Cayley(G,\Set{S})$.  The $\Set{S}$-edges are left unchanged.

If $c$ is a path in $\Cayley(G,\Set{S})$ and $M > 0$,
the \emph{$M$--saturation} of $c$ is the union of $c$ together with
all peripheral cosets $gP$ such that $d_{\Set{S}}(c,gP) < M$.

We will need to use several results due to \Drutu--Sapir
\cite{DrutuSapirTreeGraded}
describing the geometry of the word metric for a finitely
generated relatively hyperbolic group.  These results are collected below.

The first result states that neighborhoods of peripheral subgroups
are quasiconvex.

\begin{thm}[Lemma~4.15, \cite{DrutuSapirTreeGraded}]
\label{thm:PeripheralQuasiconvex}
Let $(G,\PP)$ be a relatively hyperbolic group with finite
generating set $\Set{S}$.
For each $A_0$ there is a constant $A_1=A_1(A_0)$ such that the following
holds in $\Cayley(G,\Set{S})$.
Let $c$ be a geodesic segment whose endpoints lie in the $A_0$--neighborhood
of a peripheral subgroup $P \in \PP$.
Then $c$ lies in the $A_1$--neighborhood of $P$.
\end{thm}

In particular, each peripheral subgroup $P \in \PP$ is quasiconvex
with respect to $\Set{S}$.
It follows easily that conjugates of peripheral subgroups
are quasiconvex as well.

\begin{cor}
\label{cor:PeripheralConjugateQC}
For each $P \in \PP$ and $g \in G$, the subgroup $gPg^{-1}$
is quasiconvex with respect to $\Set{S}$.
\end{cor}

\begin{proof}
Translating the result of Theorem~\ref{thm:PeripheralQuasiconvex}
by $g$, we see that neighborhoods of $gP$ are quasiconvex.
Now the Hausdorff distance between $gP$ and $gPg^{-1}$ is bounded above
by $A_0:=\abs{g}_{\Set{S}}$.
Thus any geodesic $c$ with endpoints in $gPg^{-1}$ lies in the
$A_1$--neighborhood of $gP$, for $A_1 = A_1(A_0)$.
Hence $c$ lies in the $(A_1 + A_0)$--neighborhood of $gPg^{-1}$.
\end{proof}

Quasiconvexity of peripheral subgroups
(and their conjugates) has the following
consequences using well-known results of Short \cite{Short91}.
We remark that finite generation of peripheral subgroups
was first proved by Osin \cite{Osin06} prior to the work of \Drutu--Sapir.

\begin{cor}[Osin, \Drutu--Sapir]\label{cor:PeripheralUndistorted}
If $(G,\PP)$ is relatively hyperbolic with a finite generating set,
then for each $P \in \PP$ and $g \in G$,
the conjugate $gPg^{-1}$ has a finite generating set $\Set{T}$.
Furthermore each conjugate $gPg^{-1}$ is undistorted in the sense that
the inclusion $gPg^{-1} \inclusion G$ induces a quasi-isometric embedding
\[
   \Cayley(gPg^{-1},\Set{T}) \to \Cayley(G,\Set{S}).
\]
\end{cor}

The next result states that a (sufficiently large) saturation
of a geodesic is quasiconvex.  Furthermore,
neighborhoods of saturations are also
quasiconvex.

\begin{thm}[Thm.~1.12(4), \cite{DrutuSapirTreeGraded}]
\label{thm:DSMorse}
Let $(G,\PP)$ be a relatively hyperbolic group
with finite generating set $\Set{S}$.
For each $\epsilon>0$,
there is a constant $M=M(\epsilon)$ such that the following holds.
Let $c$ be an $\epsilon$--quasigeodesic in $\Cayley(G,\Set{S})$
and let $\hat{c}$ be a geodesic
in $\Cayley(G,\Set{S} \cup \Set{P})$ with the same endpoints as $c$.
If $\tilde{c}$ is any lift of $\hat{c}$
then $\tilde{c}$ lies in the $M$--neighborhood of the $M$--saturation
of $c$ with respect to the metric $d_{\Set{S}}$.
\end{thm}

\begin{thm}[Theorem~4.1, \cite{DrutuSapirTreeGraded}]
\label{thm:Isolated}
Suppose $(G,\PP)$
is relatively hyperbolic with finite generating set $\Set{S}$.
For each $M<\infty$ there is a constant $\iota=\iota(M)<\infty$
so that for any two peripheral cosets $gP \ne g'P'$ we have
\[
   \diam \bigl( \nbd{gP}{M} \cap \nbd{g'P'}{M} \bigr) < \iota
\]
with respect to the metric $d_{\Set{S}}$.
\end{thm}

The preceding result has the following stronger form that restricts
the possible interactions
between two distinct peripheral cosets in a saturation of
a quasigeodesic.

\begin{prop}[Lemma~8.11, \cite{DrutuSapirTreeGraded}]
\label{prop:TwoPeripherals}
Let $(G,\PP)$
be a relatively hyperbolic group with finite generating set $\Set{S}$.
For each choice of positive constants
$\epsilon$, $\nu$, and $\tau$, there is a constant
$\eta_0 = \eta_0(\epsilon,\nu,\tau)$ such that the following holds.
Let $c$ be an $\epsilon$--quasigeodesic in $\Cayley(G,\Set{S})$.
Suppose $gP$ and $g'P'$ are distinct peripheral cosets
in $\Sat_\nu(c)$.
Then
\[
   \nbd{gP}{\tau} \cap \nbd{g'P'}{\tau} \subseteq \nbd{c}{\eta_0}
\]
with respect to the metric $d_{\Set{S}}$.
\end{prop}

\begin{lem}
\label{lem:NearOnePeripheral}
Let $(G,\PP)$
be a relatively hyperbolic group with finite generating set $\Set{S}$.
Choose positive constants $\epsilon$, $\nu$, and~$\tau$.
Then there exists a constant
$\eta_1=\eta_1(\epsilon,\nu,\tau)$ such that the following holds
in the graph $\Cayley(G,\Set{S})$.
Let $b$ be an $\epsilon$--quasigeodesic,
and suppose $C$ is a connected subset of $\bignbd{\Sat_\nu(c)}{\tau}$.
If $C$ does not intersect $\nbd{c}{\eta_1}$, then
$C$ intersects $\nbd{gP}{\tau}$ for a unique peripheral coset
$gP \subseteq \Sat_\nu(c)$.
Furthermore, $C \subseteq \nbd{gP}{\tau}$. 
\end{lem}

\begin{proof}
By hypothesis, $C$ is contained in the following union of open sets:
\[
   \nbd{c}{\tau} \cup
     \left( \bigcup \bigset{\nbd{gP}{\tau}}{gP\subseteq \Sat_\nu(c)} \right)
\]
Assume that $C$ does not intersect $\nbd{c}{\tau}$.
If $C$ intersects $\nbd{gP}{\tau} \cap \nbd{g'P'}{\tau}$
for distinct peripheral cosets $gP$ and $g'P'$ contained in
$\Sat_\nu(c)$, then Proposition~\ref{prop:TwoPeripherals} gives a constant
$\eta_0=\eta_0(\epsilon,\nu,\tau)$
such that $C$ intersects $\nbd{c}{\eta_0}$.
Otherwise, $C$ intersects $\nbd{c}{\tau}$ for a unique peripheral coset
$gP \subseteq \Sat_\nu(c)$, and thus $C \subseteq \nbd{gP}{\tau}$.
\end{proof}

\begin{lem}
\label{lem:cHatNearc}
Let $(G,\PP)$ be a relatively hyperbolic group with a finite generating set
$\Set{S}$. For each $\epsilon>0$
there is a constant $A_0=A_0(\epsilon)$ such that the following holds.
Let $c$ be an $\epsilon$--quasigeodesic segment in $\Cayley(G,\Set{S})$,
and let $\hat{c}$ be a geodesic segment in $\Cayley(G,\Set{S} \cup \Set{P})$.
Suppose $c$ and $\hat{c}$ have the same endpoints in $G$.
Then each vertex of $\hat{c}$ lies in the $A_0$--neighborhood of some vertex
of $c$ with respect to the metric $d_{\Set{S}}$.
\end{lem}

\begin{proof}
Let $\tilde{c}$ be a lift of $\hat{c}$ to $\Cayley(G,\Set{S})$.
By Theorem~\ref{thm:DSMorse}, $\tilde{c}$ lies in the $M$--neighborhood
of the $M$--saturation of $c$ for some $M$ depending only on $\epsilon$ and
the generating set $\Set{S}$.
Let $\iota = \iota(M)$ be the constant given by Theorem~\ref{thm:Isolated}.
Let $v$ be a vertex of $\hat{c}$.
If $v$ is within an $\Set{S}$--distance $2\iota M$
of an endpoint of $\tilde{c}$, then we are done.
Otherwise let $\gamma$ be the subpath of $\tilde{c}$
of $\Set{S}$--length $4\iota M$ centered at $v$.
Since $\gamma$ is connected,
either $\gamma$ lies in the $M$--neighborhood of some coset
$gP \subseteq \Sat_M(c)$, or $\gamma$ intersects $\nbd{c}{\eta_1}$,
where $\eta_1 = \eta_1(1,M,M)$ is given by
Lemma~\ref{lem:NearOnePeripheral}.

We will show that
$\gamma$ intersects $\nbd{c}{\eta_1}$.
Assume by way of contradiction that
$\gamma \subseteq \nbd{gP}{M}$ for $gP \subseteq \Sat_M(c)$.
It follows that the endpoints $x$ and $y$ of $\gamma$ are connected
by a relative geodesic $\hat{\gamma}$ passing through $v$,
each of whose vertices lies within an $\Set{S}$--distance $M$
of $gP$.
Evidently $\hat{\gamma}$ contains at most $2M+1$ edges.
If an edge $e$ of $\hat{\gamma}$ does not have both endpoints in $gP$,
then its endpoints are separated by an $\Set{S}$--distance less than
$\iota$ by Theorem~\ref{thm:Isolated}.
At most one edge of $\hat{\gamma}$ has both endpoints in $gP$.
Without loss of generality, we may assume the the subpath of $\hat{\gamma}$
from $x$ to $v$ does not contain such an edge.
Since $d_{\Set{S} \cup \Set{P}} (x,v) \le 2M$, it follows that
$d_{\Set{S}} (x,v) < 2\iota M$, contradicting our choice of path $\gamma$.

Thus $\gamma$ intersects $\nbd{c'}{\eta_1}$ as desired.
It follows that $v$ is within an $\Set{S}$--distance $2\iota M + \eta_1$
of $c$, completing the proof.
\end{proof}

\begin{defn}
Let $c$ be a geodesic of $\Cayley(G,\Set{S})$,
and let $\epsilon,R$ be positive constants.
A point $x \in c$ is \emph{$(\epsilon,R)$--deep}
in a peripheral left coset $gP$ (with respect to $c$)
if $x$ is not within a distance $R$ of an endpoint of $c$ and
$\ball{x}{R} \cap c$ lies in $\nbd{gP}{\epsilon}$.
If $x$ is not $(\epsilon,R)$--deep in any peripheral left coset $gP$
then $x$ is an \emph{$(\epsilon,R)$--transition point} of $c$
\end{defn}

\begin{lem}\label{lem:DisjointDeepSegments}
Let $(G,\Set{S})$
be relatively hyperbolic with finite generating set $\Set{S}$.
For each $\epsilon$ there is a constant $R=R(\epsilon)$
such that the following holds.
Let $c$ be any geodesic of $\Cayley(G,\Set{S})$,
and let $\bar{c}$ be a connected component of the set of all
$(\epsilon,R)$--deep points of $c$.
Then there is a peripheral left coset $gP$ such that each $x \in c$
is $(\epsilon,R)$--deep in $gP$ and is not $(\epsilon,R)$--deep in any other
peripheral left coset.
\end{lem}

\begin{proof}
Let $gP$ and $g'P'$ be distinct left cosets of peripheral subgroups.
Then $\nbd{gP}{\epsilon} \cap \nbd{g'P'}{\epsilon}$.
has diameter less than $R$, where $R:=\iota(\epsilon)$
is the constant given by Theorem~\ref{thm:Isolated}.

Let $c_{gP} := c \cap \nbd{gP}{\epsilon}$. (Note that this set might not be
connected.)
Deleting the points that are within a distance $R$ of an endpoint
of $c_{gP}$ gives precisely the set of points of $c$ that are
$(\epsilon,R)$--deep in $gP$.
Evidently the intersection $c_{gP} \cap c_{g'P'}$ has diameter at most
$R$ since it lies in $\nbd{gP}{\epsilon} \cap \nbd{g'P'}{\epsilon}$.
It is clear that if $x$ is $(\epsilon,R)$--deep in $gP$
and $y$ is $(\epsilon,R)$--deep in $g'P'$ then $d(x,y) \ge R$.
In particular, within each component $\bar{c}$ of the set of
$(\epsilon,R)$--deep points,
every point is $(\epsilon,R)$--deep in a unique peripheral left coset $gP$
that depends only on the choice of component $\bar{c}$.
\end{proof}

If $\bar{c}$ is a component of the $(\epsilon,R)$--deep points as above,
we say that $\bar{c}$ is \emph{$(\epsilon,R)$--deep in $gP$},
where $gP$ is the unique 
peripheral left coset associated to every point of $\bar{c}$.

\begin{lem}\label{lem:ShortPeripheralGeodesic}
For each $\epsilon,L>0$, there is a constant $\eta=\eta(\epsilon,L)$
so that the following holds.
Suppose $c$ is a geodesic in $\Cayley(G,\Set{S})$ that lies in
$\nbd{gP}{\epsilon}$ for some left coset of a peripheral subgroup $P \in \PP$.
Suppose $\hat{c}$ is a geodesic in $\Cayley(G,\Set{S} \cup \Set{P})$
such that each point of $c$ is within an $\Set{S}$--distance $L$
of some vertex of $\hat{c}$.
Then $c$ has length at most $\eta$.
\end{lem}

\begin{proof}
Let $v_0,\dots,v_n$ be the vertices of $\hat{c}$.
Without loss of generality we can assume that $v_0$ and $v_n$
both lie within an $\Set{S}$--distance $L$ of $c$.
(If not, then we first shorten $\hat{c}$ so that this condition holds.)
Let $c_i \subseteq c$ be the set $\nbd{v_i}{L} \cap c$.
Then $c$ is covered by the sets $c_0,\dots,c_n$, each of which has diameter
at most $2L$.

In order to bound the length of $c$, it suffices to bound the integer $n$
in terms of $\epsilon$ and $L$.
But $v_0$ and $v_n$ are each within an $\Set{S}$--distance $\epsilon + L$
of a vertex of $gP$.
Thus $d_{\Set{S} \cup \Set{P}} (v_0,v_n) \le 2(\epsilon + L) + 1$.
In particular, $n$ is at most $2(\epsilon +L) + 1$,
completing the proof.
\end{proof}

\begin{defn}
Let $c$ be a geodesic in a space $X$.
If $c_0$ is any subset of $c$, then the \emph{hull} of $c_0$ in $c$,
denoted $\Hull_c(c_0)$ is the smallest connected subspace of $c$
containing $c_0$.
\end{defn}

\begin{prop}
\label{prop:VerticesNearTransitions}
Let $(G,\PP)$ be relatively hyperbolic with a finite generating set $\Set{S}$.
There exist constants $\epsilon$, $R$, and $L$ such that the following holds.
Let $c$ be any geodesic of $\Cayley(G,\Set{S})$ with endpoints in $G$,
and let $\hat{c}$ be a geodesic of $\Cayley(G,\Set{S} \cup \Set{P})$
with the same endpoints as $c$.
Then in the metric $d_{\Set{S}}$, the set of vertices of $\hat{c}$
is at a Hausdorff distance at most $L$ from the set of
$(\epsilon,R)$--transition points of $c$.
Furthermore, the constants $\epsilon$ and $R$ satisfy the conclusion
of Lemma~\ref{lem:DisjointDeepSegments}
\end{prop}

\begin{proof}
Let $A_0=A_0(1)$ be the constant given by Lemma~\ref{lem:cHatNearc}.
Then each vertex of $\hat{c}$ is within a distance $A_0$ of $c$.
For each vertex $v$ of $\hat{c}$, let
\[
   c_v := \Hull_c \bigl( c \cap \ball{v}{A_0} \bigr).
\]
For each edge $e$ of $\hat{c}$ with endpoints $v$ and $w$, let
\[
   c_e := \Hull_c (c_v \cup c_w).
\]
Then $c$ is covered by the sets $c_e$ for all edges $e$ of $\hat{c}$.

If $e$ is labelled by an edge of $\Set{S}$, then $c_e$ has length
at most $2A_0+1$.
Thus every point of $c_e$ lies within a distance $2A_0+1$
of $c \cap \ball{v}{A_0}$, and hence lies within a distance $3A_0 +1$ of $v$,
where $v$ is a vertex incident to $e$.

On the other hand, if $e$ is labelled by an edge of $\Set{P}$
then the endpoints of $e$ lie in the same left coset $gP$ of some peripheral 
subgroup.
Thus $c \cap \ball{v}{A_0}$ lies in the $A_0$--neighborhood of $gP$
for each endpoint $v$ of $e$.
As each point of $c_e$ lies between two such points,
it follows that $c_e$ lies in the $A_1$--neighborhood of $gP$
where $A_1= A_1(A_0)$ is the constant given by
Proposition~\ref{thm:PeripheralQuasiconvex}.
Let $R:=R(A_1)$ be the constant given by Lemma~\ref{lem:DisjointDeepSegments}.
The points of $c_e$ within a distance $R$
of the endpoints of $c_e$
are also within a distance $A_0 + R$ of the vertices of $\hat{c}$.
All other points of $c_e$ must be $(A_1,R)$--deep points of $c$.
In particular, we have shown that
the $(A_1,R)$--transition points of $c$ each
lie within a distance $3A_0 +R + 1$ of the set of vertices of $\hat{c}$.

In order to complete the proof, we need to bound the distance from an
arbitrary vertex of $\hat{c}$ to the set of $(A_1,R)$--transition points
of $c$.
Choose a vertex $v \in \hat{c}$. If $c_v$ contains an $(A_1,R)$--transition
point, then we are done, since $c_v$ has length at most $2A_0$
and intersects $\ball{v}{A_0}$.
It suffices to assume that $c_v$ is contained in some
$(A_1,R)$--deep component $\bar{c}$ of $c$ that is deep in a 
peripheral left coset $gP$.
By Lemma~\ref{lem:DisjointDeepSegments}, each endpoint of $\bar{c}$ is
an $(A_1,R)$--transition point of $c$,
so we need only bound the distance from $c_v$ to an endpoint of $\hat{c}$.

Since $\hat{c}$ is a geodesic of $\Cayley(G,\Set{S} \cup \Set{P})$,
it contains at most one edge $e_0$
whose endpoints both lie in the coset $gP$.
Choose an edge path
$e_1,\dots,e_k$ in $\hat{c}$ such that $v$ is the initial vertex of $e_1$,
such that for $i=1\dots,k-1$ we have
$c_{e_i} \subseteq \bar{c}$, and such that
$c_k$ contains an endpoint of $\bar{c}$.
Since $\bar{c}$ has two endpoints, we can also choose $e_1,\dots,e_k$
so that none of the edges is equal to $e_0$ (if there is such an edge $e_0$).

Let $w$ denote the common endpoint of $e_{k-1}$ and $e_k$.
Both $v$ and $w$ lie within a $d_{\Set{S}}$--distance
$A_0 + A_1$ of $gP$, so $d_{\Set{S} \cup \Set{P}}(v,w) < 2A_0 + 2A_1 + 1$.
Therefore $k-1\le 2A_0 + 2A_1 + 1$.
Suppose $i = 1,\dots,k-1$.
If $e_i$ is labeled by an element of $\Set{S}$ then, as noted above,
$c_{e_i}$ has length at most $2A_0 + 1$.
On the other hand, if $e_i$ is labelled by an element of $\Set{P}$,
then the endpoints of $e_i$ lie in a peripheral left coset $g'P'\ne gP$.
Since $c_{e_i} \subseteq \bar{c}$ cannot contain any points that are
$(A_1,R)$--deep in $g'P'$,
it follows that $c_{e_i}$ has length less than $2R$.
Thus we have an upper bound on the $d_{\Set{S}}$--distance from $v$ to $w$.

Now consider the last edge $e_k$ of our edge path, which contains
an endpoint of $\bar{c}$.
If $e_k$ is labelled by an edge of $\Set{S}$, then we are done, since
the $d_{\Set{S}}$--distance from $w$ to this endpoint is bounded above by
the length of $e_k$, which is at most $2A_0 + 1$.
If $e_k$ is labelled by an edge of $\Set{P}$, then $\bar{c}$ can
intersect $c_{e_k}$ in a subsegment of length at most $R$,
so we see that $w$ is within a $d_{\Set{S}}$--distance $R$ of an endpoint
of $\bar{c}$,
completing the proof.
\end{proof}

It is worth noting the following corollary, which could also be derived
directly from the work of \Drutu--Sapir \cite{DrutuSapirTreeGraded}.
The corollary deals with only the geometry of the finite generating set
$\Set{S}$.

\begin{cor}
Let $(G,\PP)$
be a relatively hyperbolic group with finite generating set $\Set{S}$.
Then there exist constants $\epsilon,R,M$ such that $\epsilon$ and $R$
satisfy the conclusion of Lemma~\ref{lem:DisjointDeepSegments}
and such that the following holds.
Let $c$ and $c'$ be two geodesics in $\Cayley(G,\Set{S})$
with the same endpoints in $G$.
Then the set of $(\epsilon,R)$--transition points of $c$
and the set of $(\epsilon,R)$--transition points of $c'$
are at a Hausdorff distance at most $M$.
\qed
\end{cor}

We also recover the following result due to Osin
\cite[Proposition~3.15]{Osin06}.
We remark that if $\Set{S}$ is a relative generating set rather than
a generating set, the result still holds and can be derived as a
corollary of Osin \cite[Proposition~3.2]{Osin07_Peripheral}

\begin{cor}
Let $(G,\PP)$
be a relatively hyperbolic group with finite generating set $\Set{S}$.
Then there exists a constant $N$ such that if $c$ and $c'$ are two
geodesics in $\Cayley(G,\Set{S} \cup \Set{P})$
with the same endpoints in $G$, then the set of vertices
of $c$ and the set of vertices of $c'$ are within a Hausdorff distance $N$
in the metric $d_{\Set{S}}$. \qed
\end{cor}

Finally, we deduce the following corollary
describing relatively quasiconvex subgroups of $G$
using only the geometry of $\Cayley(G,\Set{S})$.

\begin{cor}
\label{cor:RelQCinCayleyGraph}
Let $(G,\PP)$ be relatively hyperbolic with finite generating set $\Set{S}$.
There are $\epsilon,R$ satisfying the conclusion of
Lemma~\ref{lem:DisjointDeepSegments} such that the following holds.
Let $H$ be a subgroup of $G$.
Then $H$ is relatively quasiconvex if and only if
there is a constant $\kappa$
such that for each geodesic $c$ in $\Cayley(G,\Set{S})$
joining points of $H$,
the set of\/ $(\epsilon,R)$--transition points of $c$
lies in the $\kappa$--neighborhood of $H$
\textup{(}with respect to the metric $d_{\Set{S}}$\textup{)}.
\end{cor}

\begin{proof}
A subgroup $H\le G$ is relatively
quasiconvex if and only if it satisfies (QC-5)
with respect to the generating set $\Set{S}$ and the proper
left invariant metric $d_{\Set{S}}$.
The corollary now follows immediately from
Proposition~\ref{prop:VerticesNearTransitions}.
\end{proof}

\section{Applications of relative quasiconvexity}
\label{sec:applications}

The present section is a collection of various basic properties
of relatively quasiconvex subgroups.  In particular, we prove
Theorems \ref{thm:QCBasicProps} and~\ref{thm:UndistortedStatement}
and Corollary~\ref{cor:GeomFiniteManifolds}.
We also briefly examine strongly relatively quasiconvex subgroups,
which were introduced by Osin in \cite{Osin06}.

The following theorem states that relatively quasiconvex subgroups
are relatively hyperbolic, with the obvious peripheral structure.

\begin{thm}[Relatively quasiconvex $\Longrightarrow$ relatively hyperbolic]
\label{thm:RelQCisRelHyp}
Let $(G,\PP)$ be relatively hyperbolic, and let $H \le G$ be a
relatively quasiconvex subgroup.
Consider the following collection of subgroups of $H$\textup{:}
\[
   \bar{\mathbb{O}} :=  \set{H \cap gPg^{-1}}{%
       \text{$g \in G$, $P \in \PP$, and $H \cap gPg^{-1}$ is infinite}}.
\] 
Then the elements of\/ $\bar{\mathbb{O}}$ lie in only finitely many
conjugacy classes in $H$.
Furthermore, if\/ $\mathbb{O}$ is a set of representatives of these conjugacy
classes then $(H,\mathbb{O})$ is relatively hyperbolic.
\end{thm}

The peripheral structure $\mathbb{O}$ constructed above is referred to as an
\emph{induced peripheral structure} on $H$ coming from $(G,\PP)$.
Note that the only ambiguity in the construction of $\mathbb{O}$
is the choice of representatives of the conjugacy classes from
$\bar{\mathbb{O}}$.

\begin{proof}
We give a dynamical
proof based on geometrically finite convergence group actions.
Let $(G,\PP)$ act geometrically finitely on a compactum $M$.
Since $H$ is (QC-1), the induced action of $H$ on
$\Lambda(H) \subseteq M$ is geometrically finite.

By a result of Tukia, a point cannot be both a conical limit point
and a parabolic point for the same convergence action
\cite[Theorem~3A]{Tukia98}.
Every conical limit point for $H$ acting on $\Lambda(H) \subseteq M$
is a conical limit point for $H$ acting on $M$, and hence also
a conical limit point for $G$ acting on $M$.
On the other hand, every parabolic point for $H$ acting on $\Lambda(H)$
is clearly a parabolic point for $G$ acting on $M$.
Thus the parabolic points of $H$ are precisely the parabolic points of $G$
that lie in $\Lambda(H)$.

Consequently the set $\bar{\mathbb{O}}$ is the set of 
maximal parabolic subgroups for the action of $H$ on $\Lambda(H)$.
Tukia has shown that a geometrically finite group action has
only finitely many conjugacy classes of maximal parabolic subgroups
\cite[Theorem~1B]{Tukia98}.
Thus $(H,\mathbb{O})$ satisfies (RH-1).
\end{proof}

A group $G$ is \emph{slender} if every subgroup $H \le G$ is finitely generated.
The following corollary is an easy consequence of the preceding theorem.

\begin{cor}
Let $(G,\PP)$ be relatively hyperbolic.
The following conditions are equivalent:
\begin{enumerate}
\item \label{item:slender} Every $P \in \PP$ is slender.
\item \label{item:QCFG} Every relatively quasiconvex subgroup
$H$ of $(G,\PP)$ is finitely generated.
\end{enumerate}
\end{cor}

\begin{proof}
(\ref{item:slender}) $\Longrightarrow$~(\ref{item:QCFG}):
If $H$ is relatively quasiconvex, then it is relatively hyperbolic
with respect to subgroups of the $P \in \PP$.
If each such $P$ is slender, then $H$ is relatively hyperbolic with respect to
finitely generated groups.
In particular $H$ is finitely generated relative to finitely many
finitely generated subgroups.
Thus $H$ is finitely generated.

(\ref{item:QCFG}) $\Longrightarrow$~(\ref{item:slender}):
Every subgroup $H$ of a peripheral subgroup $P \in \PP$
is relatively quasiconvex.
\end{proof}

Since a relatively quasiconvex subgroup $H$ of a relatively hyperbolic group $G$
is itself relatively hyperbolic,
we could in principle consider nested sequences 
$H_0 \le H_1 \le \cdots \le H_\ell =G$
such that $H_i$ is relatively quasiconvex in $H_{i+1}$.
The following corollary shows that any subgroup $H_0$ produced in this manner
is relatively quasiconvex in the original group $G$.
Thus no new subgroups arise in this manner.

\begin{cor}[Nested relative quasiconvexity]
\label{cor:NestedQC}
Suppose $(G,\PP)$ is relatively hyperbolic, and $K \le G$ is relatively
quasiconvex in $(G,\PP)$ with induced peripheral structure $\mathbb{O}$.
Suppose $H \le K$.
Then $H$ is relatively quasiconvex in $(K,\mathbb{O})$
if and only if $H$ is relatively quasiconvex in $(G,\PP)$.
\end{cor}

\begin{proof}
Let $(G,\PP)$ act geometrically finitely on a compactum $M$.
Then the limit set $\Lambda(H)$ for the action of $H$ on $M$ is the same as
the limit set for the restricted action of $H$ on $\Lambda(K)$.
The property of geometrical finiteness of the action of $H$ on $\Lambda(H)$
is intrinsic to $\Lambda(H)$.
Thus $H$ satisfies (QC-1) as a subgroup of $(G,\PP)$
if and only if $H$ satisfies (QC-1) as a subgroup of $(K,\mathbb{O})$.
\end{proof}

Using Lemma~\ref{lem:cHatNearc}, the proof of
Theorem~\ref{thm:UndistortedStatement} is straightforward.

\begin{proof}[Proof of Theorem~\ref{thm:UndistortedStatement}]
Since $\Cayley(H,\Set{T})$ is a geodesic space, it follows that each pair
of points in $H$ is connected in $\Cayley(G,\Set{S})$ by an
$\epsilon$--quasigeodesic
that lies in the $\epsilon$--neighborhood of $H$.
Let $c$ be such a quasigeodesic.
By Lemma~\ref{lem:cHatNearc}, if $\hat{c}$ is any geodesic
of $\Cayley(G,\Set{S} \cup \Set{P})$ with the same endpoints as $c$,
then each vertex $v$ of $\hat{c}$ lies within a $d_{\Set{S}}$--distance
$A_0$ of $c$, where $A_0=A_0(\epsilon)$.
Consequently, $v$ is within a $d_{\Set{S}}$--distance $A_0 + \epsilon$
of some point of $H$,
establishing that $H$ satisfies definition (QC-5) with respect to
the generating set $\Set{S}$ and the proper, left invariant metric
$d_{\Set{S}}$.
\end{proof}

The following result roughly states that an element of a group that lies close
to two cosets $xH$ and $yK$ also lies near the intersection
$xHx^{-1} \cap yKy^{-1}$.

\begin{prop}\label{prop:CloseCosets}
Let $G$ have a proper, left invariant metric $d$, and suppose $xH$ and $yK$
are arbitrary left cosets of subgroups of~$G$.
For each constant $L$ there is a constant $L'=L'(G,d,xH,yK)$
so that in the metric space $(G,d)$ we have
\[
   \nbd{x H}{L} \cap \nbd{y K}{L} \subseteq
     \nbd{x H x^{-1} \cap y K y^{-1}}{L'}.
\]
\end{prop}

\begin{proof}
If there is no such $L'$, then there is a sequence $(z_i)$ in~$G$ so that
$z_i$ is in the $L$--neighborhood of both $xH$ and $yK$,
but $i < d(z_i, xHx^{-1} \cap yKy^{-1})$ for each~$i$.
It follows that $z_i = x h_i p_i = y k_i q_i$
for some $h_i\in H$, $k_i \in K$ and $p_i,q_i \in G$
with $d(1,p_i)$ and $d(1,q_i)$ both less than~$L$.
Since the ball of radius $L$ in $(G,d)$ is finite,
we can pass to a subsequence in which $p_i$
and $q_i$ are constants $p$ and $q$.
Then for each $i$ we can express $z_i = x h_i p = y k_i q$.
Therefore
\[
   z_i z_1^{-1} = x h_i h_1^{-1} x^{-1} = y k_i k_1^{-1} y^{-1}
     \in xHx^{-1} \cap yKy^{-1}.
\]
It follows that the distance between $z_i$
and $xHx^{-1} \cap yKy^{-1}$ is at most
$d(1,z_1)$ for all~$i$, contradicting our choice of $(x_i)$.
\end{proof}

\begin{cor}\label{cor:Intersections}
Let $H$ and $K$ be relatively quasiconvex subgroups of a relatively
hyperbolic group $(G,\PP)$.
Then $H \cap K$ is relatively quasiconvex in $(G,\PP)$.
\end{cor}

\begin{proof}
Choose a finite relative generating set $\Set{S}$ for $(G,\PP)$
and a proper, left invariant metric $d$ for $G$.
Consider the Cayley graph $\bar\Gamma:= \Cayley(G,\Set{S} \cup \Set{P})$.
If $c$ is a geodesic in $\bar\Gamma$ with both endpoints in $H \cap K$,
then each vertex $v$ of $c$ lies within a uniformly bounded $d$--distance
of both $H$ and $K$ by (QC-5).
Therefore, by Proposition~\ref{prop:CloseCosets}
we also have a uniform bound on the distance $d(v,H \cap K)$.
Thus $H \cap K$ satisfies (QC-5) as well.
\end{proof}

Theorem~\ref{thm:RelQCisRelHyp} and Corollary~\ref{cor:Intersections}
together complete the proof of Theorem~\ref{thm:QCBasicProps}.
The proof of Corollary~\ref{cor:GeomFiniteManifolds} is
now immediate.

\begin{proof}[Proof of Corollary~\ref{cor:GeomFiniteManifolds}]
Let $H$ be a subgroup of a geometrically finite group $G \le \Isom(X)$,
where $X$ is a finite volume manifold with pinched negative curvature.
The first assertion of the Corollary follows from
Definition~(QC-1) for relative quasiconvexity.
The second assertion is an immediate consequence of
Corollary~\ref{cor:Intersections}.
\end{proof}

The notion of a strongly relatively quasiconvex subgroup was introduced by Osin
in \cite[Section~4.2]{Osin06} in the special case that $G$ is finitely generated.  Osin's results about such subgroups extend easily to the general case.
We give the definition and then list several basic results about
strongly relatively hyperbolic subgroups.

\begin{defn}
Let $(G,\PP)$ be relatively hyperbolic.
A subgroup $H \le G$ is \emph{strongly relatively quasiconvex}
if $H$ is relatively quasiconvex and the induced peripheral structure
$\mathbb{O}$ on $H$ is empty.
In other words, the subgroup $H \cap gPg^{-1}$ is finite
for all $g \in G$ and $P \in \PP$.
\end{defn}

\begin{thm}[See Theorem~4.16, \cite{Osin06}]
\label{thm:StrongQCisHyp}
If $H$ is strongly relatively quasiconvex in $(G,\PP)$ then $H$
is finitely generated and word hyperbolic.
\end{thm}

\begin{proof}
The result is immediate, since $(H,\emptyset)$ is relatively hyperbolic.
\end{proof}

We also obtain the following result, proved by Osin 
in the case that $G$ is finitely generated \cite[Proposition~4.18]{Osin06}.
(We remark that Osin erroneously claims to have proven
Corollary~\ref{cor:Intersections},
although his proof actually gives the result below.)

\begin{thm}
Let $H$ and $K$ be subgroups of a relatively hyperbolic group $(G,\PP)$.
If $H$ is strongly relatively quasiconvex and $K$ is relatively
quasiconvex, then $H \cap K$ is strongly relatively quasiconvex.
\end{thm}

\begin{proof}
Since $H \cap K$ is a relatively quasiconvex subgroup of $(H,\emptyset)$,
its induced peripheral structure must be empty.
\end{proof}

The following
theorem characterizes strongly relatively quasiconvex
subgroups in several ways.

\begin{thm}
Let $(G,\PP)$ be a relatively hyperbolic group with finite relative generating
set $\Set{S}$.
Choose a cusp uniform action of $(G,\PP)$ on a $\delta$--hyperbolic space
$(X,\rho)$.
If $H$ is a subgroup of $G$,
then the following are equivalent.
\begin{enumerate}
\item \label{item:StrongRelQC}
$H$ is strongly relatively quasiconvex in $(G,\PP)$.
\item \label{item:UniformConvergenceGroup}
$H$ acts on $\Lambda H \subseteq \boundary X$ as a uniform convergence
group; i.e., the action of $H$ on the space of distinct
triples of points of $\Lambda H$ is proper and cocompact.
\item\label{item:AllConicalLimits}
The action of $H$ on $\Lambda H \subseteq \boundary X$
is a convergence group action, and every point of $\Lambda H$
is a conical limit.
\item\label{item:CocompactOnJoin}
Either $H$ is finite, or
the action of $H$ on $\join(\Lambda X) \subseteq X$ is cocompact.
\item\label{item:QuasiconvexinCuspedSpace}
For each basepoint $x \in X$ the orbit $Hx \subseteq X$ is quasiconvex.
\item \label{item:QIinCuspedSpace}
$H$ is generated by a finite set $\Set{T}$ and for each basepoint $x \in X$
the orbit map $h \mapsto h(x)$ induces a quasi-isometric embedding
\[
   \Cayley(H,\Set{T}) \to (X,\rho).
\]
\item \label{item:QIinRelCayley}
$H$ is generated by a finite set $\Set{T}$
and the inclusion $H \inclusion G$ induces a quasi-isometric embedding
\[
   \Cayley(H,\Set{T}) \to \Cayley(G,\Set{S} \cup \Set{P}).
\]
\end{enumerate}
\end{thm}

\begin{proof}
The equivalence of
(\ref{item:UniformConvergenceGroup}) and
(\ref{item:AllConicalLimits})
for convergence group actions has been proved by Bowditch and
also by Tukia \cite{Bowditch98TopologicalHyp,Tukia98}.
It is clear that (\ref{item:AllConicalLimits}) is simply definition
(QC-1) with no parabolic subgroups.
Similarly, (\ref{item:CocompactOnJoin}) is simply (QC-2) with no
parabolic subgroups.
Thus the first four conditions are equivalent.

(\ref{item:CocompactOnJoin})
$\Longleftrightarrow$~(\ref{item:QuasiconvexinCuspedSpace}):
Suppose $H$ acts cocompactly on $\join(\Lambda H)$.
Choose $x \in \join(\Lambda H)$.
Since $\join(\Lambda H)$ is quasiconvex, it is clear that any geodesic
joining points of $Hx$ lies near $Hx$.

Conversely, suppose $Hx$ is quasiconvex.
As in the proof of Proposition~\ref{prop:QC23},
each geodesic connecting points of $Hx$ lies near $Hx$,
so that same holds for biinfinite geodesics obtained as limits of such
geodesics.
Every point of $\Lambda H$ is a limit point of $Hx$ by definition.
Thus each pair in $\Lambda H$ is joined by such a geodesic $c$.
Any other geodesic $c'$ joining the same pair is at a Hausdorff
distance $2\delta$ from $c$.
The set $\join(\Lambda H)$ is the union of all such lines $c'$,
so $\join(\Lambda H)$ lies in a uniformly bounded neighborhood of $Hx$.
In other words, $H$ acts cocompactly on $\join(\Lambda H)$.

(\ref{item:QuasiconvexinCuspedSpace})
$\Longleftrightarrow$~(\ref{item:QIinCuspedSpace}):
This equivalence is elementary and well-known for subgroups $H$
of a word hyperbolic group $H'$ by work of Short.
Replace $H'$ with $X$ to prove the desired result.

(\ref{item:StrongRelQC}) $\Longleftrightarrow$~(\ref{item:QIinRelCayley}):
This equivalence is proved by Osin in the case that $\Set{S}$
generates $G$ \cite[Theorem~4.13]{Osin06}.
However Osin's proof never uses this hypothesis,
so it extends without change to the present setting
once we replace $d_{\Set{S}}$ with a proper, left invariant metric $d$
on $G$.
(The forward implication also follows from Theorem~\ref{thm:QCEmbedding}
below.)
\end{proof}

\section{Geometric properties of subgroup inclusion}
\label{sec:SubgroupInclusion}

In this section, we examine the geometry of the
inclusion $H \inclusion G$ where $H$ is relatively quasiconvex in $(G,\PP)$.
Theorem~\ref{thm:QCEmbedding} shows that the relative Cayley graph of $H$
embeds quasi-isometrically into the relative Cayley graph of $G$.
In Theorem~\ref{thm:Distortion} we compute the distortion of $H$ in $G$
in the case when $H$ and $G$ are both finitely generated.
We finish the section with a proof of Corollary~\ref{cor:GeomFiniteKleinian}.

If $(G,\PP)$ is relatively hyperbolic, recall that $\mathcal{P}$ 
is the disjoint union
\[
   \mathcal{P} := \coprod_{P \in \PP} (\tilde{P} - \{1\}),
\]
where $\tilde{P}$ is an abstract group isomorphic to $P$.
If $H$ is relatively quasiconvex in $(G,\PP)$ with induced peripheral structure
$\mathbb{O}$, then the set $\mathcal{O}$ is defined analogously
to $\mathcal{P}$.

\begin{thm}
\label{thm:QCEmbedding}
Let $(G,\PP)$ be relatively hyperbolic, and let $H$ be a relatively
quasiconvex subgroup with induced peripheral structure $\mathbb{O}$.
Choose finite relative generating sets $\Set{S}$ and $\Set{T}$ for $(G,\PP)$
and $(H,\mathbb{O})$ respectively.
Then the inclusion $H \inclusion G$ induces a
quasi-isometric embedding
\[
  f \colon (H,d_{\Set{T} \cup \Set{O}}) \to (G,d_{\Set{S} \cup \Set{P}}).
\]
%
\end{thm}


In principle, one could prove the theorem by modifying
Short's original proof that quasiconvex subgroups of a hyperbolic group
are finitely generated and undistorted \cite{Short91}.
However, it seems more direct to use electric geometry
since in that setting the quasi-isometric embedding is geometrically
obvious.

\begin{proof}
Let $(G,\PP)$ have a cusp uniform action on $(X,\rho)$.
Recall that $Y':=\join(\Lambda H) \subseteq X$
is $\kappa$--quasiconvex for some $\kappa>0$.
By (QC-2), we know that $(H,\mathbb{O})$ has a cusp uniform action
on a geodesic space $Y$ that is $H$--equivariantly quasi-isometric
to $Y'$ endowed with the subspace metric $\rho$.
By abuse of notation we will also use $\rho$ to refer to the metric on $Y$.

Choose a truncated space $X-U$ for $G$ such that the only horoballs
of $U$ meeting $Y'$ are those centered at parabolic points of $H$.
Pull $X-U$ back to $Y$ to get an induced truncated space $Y-V$ for $H$.
Suppose the horoballs of $U$ and also of $V$ are pairwise separated by
at least a distance $r$,
where $r$ is given by Lemma~\ref{lem:LiftIsQuasigeodesic}.
Let $\hat\rho$ denote the electric metric associated to $(X,U)$,
and also the electric metric associated to $(Y,V)$.

For any semipolygonal relative geodesic $(\gamma,\alpha)$ in $(X,U)$
connecting two points of $\join(\Lambda H)$, the path $\gamma$ is
a quasigeodesic of $(X,\rho)$ by Lemma~\ref{lem:LiftIsQuasigeodesic}.
Thus by the Morse Lemma and the quasiconvexity
of $\join(\Lambda H)$, we see that $\gamma$ lies in the $L$--neighborhood
of $\join(\Lambda H)$ for some $L$ depending only on $\kappa$ and
the choice of $U$.
It follows that $(Y,\hat\rho)$ is quasi-isometric to
$\join(\Lambda H)$ with the electric metric
$\hat\rho$ induced as a subspace of $X$.
Therefore the map $Y \to X$ induces an $H$--equivariant
quasi-isometric embedding $(Y,\hat\rho) \to (X,\hat\rho)$.

Applying Proposition~\ref{prop:ElectricQIRelCayley} to both $Y$ and $X$,
gives equivariant quasi-isometries
\[
   (H,d_{\Set{T} \cup \Set{O}}) \to (Y,\hat\rho)
   \quad \text{and} \quad
   (G,d_{\Set{S} \cup \Set{P}}) \to (X,\hat\rho).
\]
Thus the inclusion $H \inclusion G$
induces a quasi-isometric embedding
\[
   f \colon (H,d_{\Set{T} \cup \Set{O}}) \to (G,d_{\Set{S} \cup \Set{P}}),
\]
completing the proof.
\end{proof}

Two monotone functions $f,g \colon [0,\infty) \to [0,\infty)$
are said to be \emph{$\simeq$~equivalent} if $f \preceq g$ and $g \preceq f$,
where $f \preceq g$ means that there exists a constant $C>0$ such that
\[
   f(r) \le C \, g(Cr+C) + Cr + C
\]
for all $r \ge 0$.  One extends this equivalence to functions
$f \colon \N \to [0,\infty)$ by extending $f$ to be constant on each interval
$[n,n+1)$.

\begin{defn}[distortion]
If $G$ is a group with finite generating set $\Set{S}$ and $H$ is a subgroup
with finite generating set $\Set{T}$, the \emph{distortion}
of $(H,\Set{T})$ in $(G,\Set{S})$ is 
the function
\[
   \Delta_H^G(n) := \max \,\bigset{\abs{h}_{\Set{T}}}{\text{$h \in H$
              and $\abs{h}_{\Set{S}} \le n$}}.
\]
Up to $\simeq$~equivalence, this function does not depend on the choice
of the finite generating sets $\Set{S}$ and $\Set{T}$.
A subgroup $H \le G$ is \emph{undistorted} if and only if $\Delta_H^G \simeq n$.
\end{defn}

\begin{defn}
A function $f \colon \N \to \N$ is \emph{superadditive}
if
\[
   f(a+b) \ge f(a) + f(b) \quad \text{for all $a,b \in \N$.}
\]
The \emph{superadditive closure} of a function $f \colon \N \to \N$
is the smallest superadditive function $\bar{f}$ such that $\bar{f}(n) \ge f(n)$
for all $n \in \N$.
The superadditive closure is given by the formula
\[
   \bar{f}(n) := \max \,\bigset{f(n_1) + \dotsb + f(n_r)}{r \ge 1 \text{ and }
       n_1 + \dotsb + n_r = n}.
\]
See Brick \cite{Brick93} and Guba--Sapir \cite{GubaSapir99} for more information about superadditive functions.
\end{defn}

The following lemma is useful when passing from the peripheral subgroups
of $G$ to the induced peripheral subgroups of $H$.

\begin{lem}
\label{lem:InducedPeripherals}
Let $(G,\PP)$ be relatively hyperbolic, and suppose $H$ is a relatively 
quasiconvex subgroup with induced peripheral structure $\mathbb{O}$.
Fix a proper, left invariant metric $d$ on $G$.
For each constant $L>0$ there are constants $D=D(L)$ and $L' = L'(L)$
such that the following holds.

Let $\mathcal{A}$ be the finite collection of peripheral left cosets $gP$ of
$(G,\PP)$ such that $H \cap gPg^{-1} \in \mathbb{O}$.
Let $\mathcal{B}_L$ be the finite collection of peripheral left cosets $gP$
that intersect the ball $\ball{1}{L}$ in $(G,d)$ such that
$H \cap gPg^{-1}$ is finite.

Suppose $g'P$ is a peripheral left coset of $(G,\PP)$ such that $d(H,g'P) < L$.
Then $g'P = hgP$ for some $h \in H$ and some
$gP \in \mathcal{A} \cup \mathcal{B}_L$, and
\[
   \nbd{H}{L} \cap \nbd{g'P}{L} \subseteq \bignbd{h(H \cap gPg^{-1})}{L'}.
\]

Furthermore, if $gP \in \mathcal{B}_L$ then
\[
   \diam \bigl(\nbd{H}{L} \cap \nbd{g'P}{L} \bigr) < D.
\]
\end{lem}

\begin{proof}
Since $d(H,g'P) < L$, we can express $g'P$ as $h_0gP$
for some $h_0 \in H$
and some coset $gP$ intersecting $\ball{1}{L}$ in $(G,d)$.
If $H \cap gPg^{-1}$ is finite, then $gP \in \mathcal{B}_L$.
In this case, we let $h := h_0$.

On the other hand, if $H \cap gPg^{-1}$ is infinite,
then it is conjugate in $H$ to a
peripheral subgroup of $\mathbb{O}$.
In other words, there is $h_1 \in H$ such that
$H \cap h_1 g P g^{-1} h_1^{-1} \in \mathbb{O}$.
Therefore $h_1 gP = (h_1 h_0^{-1})g'P \in \mathcal{A}$.
In this case, let $h := h_1 h_0^{-1}$.
In either case we have $g'P = hgP$ for some
$gP \in \mathcal{A} \cup \mathcal{B}_L$ as desired.

Since $\mathcal{A} \cup \mathcal{B}_L$ is finite,
Proposition~\ref{prop:CloseCosets} gives is a constant $L'= L'(L)$
such that for every $gP \in \mathcal{A} \cup \mathcal{B}_L$ we have
\[
   \nbd{H}{L} \cap \nbd{gP}{L} \subseteq \nbd{H \cap gPg^{-1}}{L'}.
\]
Translating this statement by $h$ gives
\[
   \nbd{H}{L} \cap \nbd{g'P}{L} \subseteq \bignbd{ h(H \cap gPg^{-1}) }{L'}.
\]

Since $\mathcal{B}_L$ is finite, there is some $D_0 < \infty$
such that for each $gP \in \mathcal{B}_L$ we have
\[
   \diam ( H \cap gPg^{-1}) < D_0.
\]
If we let $D := D_0 + 2L'$, then
\[
   \diam \bigl( \nbd{H}{L} \cap \nbd{g'P}{L} \bigr) < D
\]
whenever $gP \in \mathcal{B}_L$.
\end{proof}

\begin{thm}
\label{thm:Distortion}
Let $(G,\PP)$ be relatively hyperbolic with a finite generating set $\Set{S}$.
Let $H$ be a relatively quasiconvex subgroup with a finite generating
set $\Set{T}$.
Let $\mathbb{O}$ denote the induced peripheral structure on $H$.
For each $O = H \cap gPg^{-1} \in \mathbb{O}$ let $\delta_O$
denote the distortion of the finitely generated group $O$
in the finitely generated group $gPg^{-1}$.
Then the distortion of $H$ in $G$ satisfies
\[
   f \preceq \Delta_H^G  \preceq \bar{f},
\]
where 
\[
   f(n) := \max_{O \in \mathbb{O}} \delta_O(n)
\]
and $\bar{f}$ is the superadditive closure of $f$.
\end{thm}

We remark that if $\mathbb{O}$ is empty (i.e., $H$ is strongly relatively
quasiconvex) then $f(n)$ is considered to be identically zero,
which is $\simeq$~equivalent to a linear function.
Thus in the case of a strongly relatively quasiconvex subgroup $H$,
the preceding theorem implies that $H$ is undistorted in $G$.

The outline of the proof is inspired by Short's proof
that quasiconvex subgroups of a finitely generated group
are finitely generated and undistorted \cite{Short91}.

\begin{proof}
First let us see that $f \preceq \Delta_H^G$.
If $\mathbb{O}$ is empty, there is nothing to show.
Thus it suffices to check that $\delta_O \preceq \Delta_H^G$ for each
$O \in \mathbb{O}$.
Choose finite generating sets
$\Set{A}$, $\Set{B}$, $\Set{S}$, and $\Set{T}$ for
$gPg^{-1}$, $O$, $G$, and $H$ respectively
such that $\Set{A} \subseteq \Set{S}$.
Choose $z \in O$ such that
\[
   \abs{z}_{\Set{A}} \le n \quad \text{and} \quad
   \abs{z}_{\Set{B}} = \delta_O(n).
\]
Since the distortion of $O$ in $H$ is linear, we have
\[
   \abs{z}_{\Set{S}} \le n \quad \text{and} \quad
   \abs{z}_{\Set{T}} > (\Delta_O^H)^{-1} \bigl(\delta_O(n) - 1\bigr)
      \simeq \delta_O(n).
\]
Thus $\delta_O \preceq \Delta_H^G$.

We will now consider the less trivial inequality
$\Delta_H^G \preceq \bar{f}$.
Fix a finite generating set $\Set{S}$ for $G$, and
let $\epsilon,R,\kappa$ be the constants given by
Corollary~\ref{cor:RelQCinCayleyGraph}.
Suppose $\abs{h}_{\Set{S}} \le n$, and let $c$ be a geodesic in
$\Cayley(G,\Set{S})$ from $1$ to $h$.
Then the set of $(\epsilon,R)$--transition points of $c$
lies in the $\kappa$--neighborhood of $H$ in $\Cayley(G,\Set{S})$.



The geodesic $c$ is a concatenation of paths
$c_0,c'_1,c_1,\dots,c'_{\ell},c_\ell$
such that each $c_i$ is a component of the set of $(\epsilon,R)$--transition
points of $c$, and each $c'_i$ is a component
of the set of $(\epsilon,R)$--deep points of $c$.
Each path $c'_i$ is $(\epsilon,R)$--deep in a unique peripheral left
coset $g'_i P_i$ by Lemma~\ref{lem:DisjointDeepSegments}.

We will assume without loss of generality that each $c_i$ and each $c'_i$
is a sequence of edges.  Increasing $R$ and $\kappa$ slightly,
we can also assume that each $c_i$ contains at least one edge,
so that $c$ has length at least $\ell+1$; in other words, $\ell+1\le n$.

Lemma~\ref{lem:InducedPeripherals} gives constants
$L'=L'(\kappa + \epsilon)$ and $D= D(\kappa+\epsilon)$ and sets
of cosets $\mathcal{A}$ and $\mathcal{B}_{\kappa+\epsilon}$
such that for each $i$ we have
$g'_i P_i = k_i {g}_i {P}_i$ for some $k_i \in H$ and some
${g}_i {P}_i \in \mathcal{A} \cup \mathcal{B}_{\kappa+\epsilon}$.
Furthermore, if we let $O_i = H \cap g_i P_i g_i^{-1}$, then
in $(G,d_{\Set{S}})$ the endpoints of $c'_i$ lie in 
$\nbd{H}{\kappa} \cap \nbd{g'_i P_i}{\epsilon}$, which is a subset of
$\nbd{k_i O_i}{L'}$.
If $b_i$ is the label on the path $c'_i$, then
we can choose elements $u_i$ and $v_i$ in $G$ such that
$\abs{u_i}_{\Set{S}}$ and $\abs{v_i}_{\Set{S}}$ are less than $L'$
and such that $o_i := u_i b_i v_i^{-1}$ is an element of $O_i$
as shown in Figure~\ref{fig:DistortionPeripheral}.

\begin{figure}[t]
\labellist
\small
\pinlabel $1$ at 15 31
\pinlabel $H$ at 44 5
\pinlabel $k_1g_1P_1$ at 108 105
\pinlabel $b_1$ at 107 77
\pinlabel $u_1$ at 92 61
\pinlabel $v_1$ at 141 61
\pinlabel $o_1$ at 120 26
\pinlabel $k_1O_1$ at 155 23
\pinlabel {$k_\ell g_\ell P_\ell$} at 254 102
\pinlabel $b_\ell$ at 247 77
\pinlabel $u_\ell$ at 211 61
\pinlabel $v_\ell$ at 267 61
\pinlabel $o_\ell$ at 253 26
\pinlabel {$k_\ell O_\ell$} at 290 23
\pinlabel $h$ at 321 31
\endlabellist
\begin{center}
\includegraphics{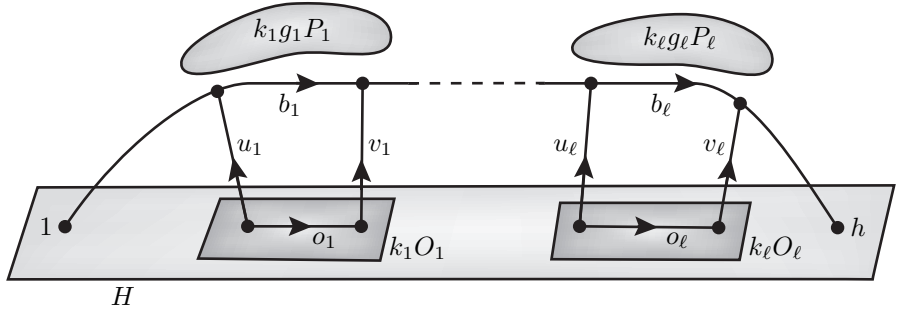}
\end{center}
\caption{The elements $o_i$ lie in the peripheral subgroup $O_i$ of $H$.}
\label{fig:DistortionPeripheral}
\end{figure}

Suppose the edges of $c_i$ are labelled by the sequence
$a_{i,1} \cdots a_{i,n_i}$ of elements of~$\Set{S}$.
Each vertex of $c_i$ lies within an $\Set{S}$--distance $\kappa$ of a vertex
of $H$.
Let $w_{0,0} = w_{\ell,n_\ell} = 1$,
and for each $i=1,\dots,\ell$
let $w_{i-1,n_{i-1}} := u_i$, and $w_{i,0} := v_i$.
Then for each $j=1,\dots,n_i-1$ there exists $w_{i,j} \in G$ 
with $\abs{w_{i,j}}_{\Set{S}} < \kappa$ such that whenever $j = 1,\dots,n_i$
the element
$h_{i,j}:= w_{i,j-1} a_{i,j} w_{i,j}^{-1}$ lies in $H$
as illustrated in Figure~\ref{fig:DistortionTransition}.

\begin{figure}[t]
\labellist
\small
\pinlabel $H$ at 66 7
\pinlabel {$v_i=w_{i,0}$} [r] at 56 64
\pinlabel $w_{i,1}$ [r] at 116 64
\pinlabel $w_{i,{n_i}-1}$ [r] at 197 64
\pinlabel {$w_{i,n_i}=u_{i+1}$} [l] at 254 64
\pinlabel $a_{i,1}$ at 86 96
\pinlabel $a_{i,n_i}$ at 225 96
\pinlabel $h_{i,1}$ at 83 23
\pinlabel $h_{i,n_i}$ at 225 23
\endlabellist
\begin{center}
\includegraphics{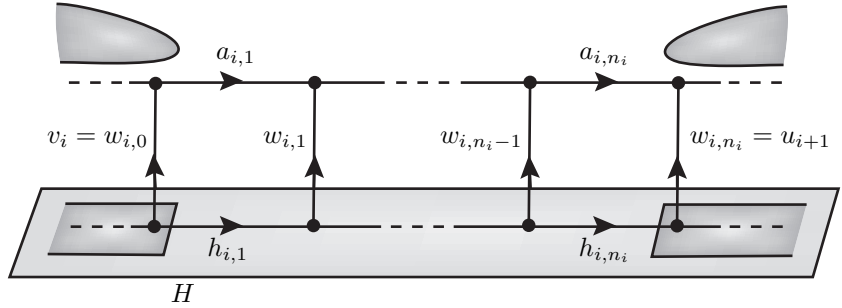}
\end{center}
\caption{The elements $h_{i,j}$ lie in $H$.}
\label{fig:DistortionTransition}
\end{figure}

We now have two decompositions of $h$.  The first in $G$ corresponds to
the labels along the geodesic $c$:
\[
   h = (a_{0,1}\cdots a_{0,n_0})b_1(a_{1,1}\cdots a_{1,n_1})
        \cdots
       b_{\ell}(a_{\ell,1}\cdots a_{\ell,n_\ell}).
\]
The second is a decomposition of $h$ in $H$, that ``tracks close'' to $c$:
\[
   h = (h_{0,1}\cdots h_{0,n_0})o_1(h_{1,1}\cdots h_{1,n_1})
        \cdots
       o_{\ell}(h_{\ell,1}\cdots h_{\ell,n_\ell}).
\]
Since $\abs{w_{i,j}}_{\Set{S}} < \kappa + L'$,
we have $\abs{h_{i,j}}_{\Set{S}} < 2(\kappa + L') + 1$.
If $g_iP_i \in \mathcal{B}_{\kappa+\epsilon}$ for some $i$,
then $\abs{b_i}_{\Set{S}} < D$ by Lemma~\ref{lem:InducedPeripherals}, so that
$\abs{o_i}_{\Set{S}} < 2L' + D$.
Let $B$ denote the finite ball in $(H,d_{\Set{S}})$ centered at $1$ with radius
$2(\kappa + L') + D + 1$,
and choose a finite generating set $\Set{T}$ for $H$
such that $B \subseteq \Set{T}$.
Then we have
\[
   \abs{h}_{\Set{S}} = n_0 + \abs{b_1}_{\Set{S}} + n_1 + \cdots
      + \abs{b_\ell}_{\Set{S}} + n_\ell \le n
\]
and
\[
   \abs{h}_{\Set{T}} \le n_0 + \abs{o_1}_{\Set{T}} + n_1 + \cdots
      + \abs{o_\ell}_{\Set{T}} + n_\ell.
\]
If $g_iP_i \in \mathcal{B}_{\kappa+\epsilon}$ then $\abs{o_i}_{\Set{T}} \le 1$
by our choice of $\Set{T}$.
On the other hand, if $g_iP_i \in \mathcal{A}$, then 
$O_i \in \mathbb{O}$.
Since the distortions of $O_i$ in $H$ and of $g_iP_ig_i^{-1}$ in $G$ are linear,
there is a constant $C$ depending only on $\mathbb{O}$ such that
\begin{align*}
   \abs{o_i}_{\Set{T}}
      &\le C \,\delta_{O_i} \bigl( C \abs{o_i}_{\Set{S}} + C \bigr)+C \\
      &\le C \,\delta_{O_i} \bigl( C 
         \bigl( \abs{u_i}_{\Set{S}}
         + \abs{b_i}_{\Set{S}} + \abs{v_i}_{\Set{S}} \bigr)
         + C \bigr) +C \\
      &< C \,\delta_{O_i} \bigl( C \abs{b_i}_{\Set{S}} + 2CL' + C \bigl) +C \\
      &\le C \, \bar{f} \bigl( C \abs{b_i}_{\Set{S}} + 2CL' + C \bigl) + C 
\end{align*}
Using the superadditivity of $\bar{f}$ and the fact that $\ell< n$, we see that
\begin{align*}
   \abs{h}_{\Set{T}}
      &< \sum_{i=0}^\ell n_i
        + \sum_{i=1}^{\ell}
        C \, \bar{f} \bigl( C \abs{b_i}_{\Set{S}} + 2CL' + C \bigl) + C \\
      &\le n + C \, \bar{f}
      \biggl( \sum_{i=1}^{\ell} C\abs{b_i}_{\Set{S}}
        +2CL' +C \biggr) + C\ell \\
      &\le n + C \, \bar{f} \bigl( Cn + (2CL'+C)\ell \bigr) +C\ell \\
      &\le n + C \, \bar{f} \bigl( (2C + 2CL')n \bigr) + Cn.
\end{align*}
Thus $\Delta_H^G \preceq \bar{f}$ as desired.
\end{proof}

We will now use the previous theorem to complete the proof 
of Corollary~\ref{cor:GeomFiniteKleinian}.

\begin{proof}[Proof of Corollary~\ref{cor:GeomFiniteKleinian}]
Properties (1) and (2) are equivalent by
Corollary~\ref{cor:GeomFiniteManifolds}.

If $H$ is undistorted in $G$, then it is relatively quasiconvex
by Theorem~\ref{thm:UndistortedStatement}.
Recall that since $G$ is a geometrically finite Kleinian group
its maximal parabolic subgroups $P$
are finitely generated and virtually abelian.
Every subgroup $O$ of such a group $P$ is finitely generated
and undistorted.
If $H$ is relatively quasiconvex in $G$ it follows from
Theorem~\ref{thm:Distortion} that $H$ is undistorted in $G$.
Thus (2) and (3) are equivalent.

A geometrically finite Kleinian group $G$ acts properly discontinuously,
cocompactly, and isometrically on a $\CAT(0)$ space $Y$ with isolated flats
(see, for instance, Hruska \cite{HruskaGeometric}).
The author proves in \cite{HruskaGeometric} that
a subgroup $H$ of a $\CAT(0)$ group $G$
is $\CAT(0)$--quasiconvex if and only if $H$ is finitely
generated and undistorted in $G$.
Thus (3) is equivalent to (4).
\end{proof}

\bibliographystyle{alpha}
\bibliography{relqc}

\end{document}